%% file: sigma23-095.tex
\pdfoutput=1
\RequirePackage{ifpdf}
\ifpdf % We~are running pdfTeX in pdf mode
\documentclass[pdftex]{sigma}
\else
\documentclass{sigma}
\fi

\usepackage{mathabx,tikz}
 \usetikzlibrary{decorations.pathreplacing,backgrounds,decorations.markings}

\numberwithin{equation}{section}

\newtheorem*{Th*}{Theorem \ref{thm:isoEDHecke}}
\newtheorem*{Th**}{Theorem~\ref{thm:isoEAHecke}}
\newtheorem*{Prop*}{Proposition~\ref{prop:formalDHecke}}
\newtheorem*{Prop**}{Proposition~\ref{prop:formalAHecke}}

\input defs.tex

\newcommand{\fdot}[3][]{ \node [anchor = center, fill=white, draw=black,circle,inner sep=2pt] at (#3) {}; \node[xshift=-.065cm, yshift=-.065cm,anchor = south west] at (#3){\small $#1$}; \node[xshift=-.065cm, yshift=.065cm,anchor = north west] at (#3){\small $#2$}; }

\begin{document}
\allowdisplaybreaks

\newcommand{\arXivNumber}{1906.03055}

\renewcommand{\PaperNumber}{095}

\FirstPageHeading

\ShortArticleName{DG-Enhanced Hecke and KLR Algebras}

\ArticleName{DG-Enhanced Hecke and KLR Algebras}

\Author{Ruslan MAKSIMAU~$^{\rm a}$ and Pedro VAZ~$^{\rm b}$}

\AuthorNameForHeading{R.~Maksimau and P.~Vaz}

\Address{$^{\rm a)}$~Laboratoire Analyse G\'eom\'etrie Mod\'elisation, CY Cergy Paris Universit\'e,\\
\hphantom{$^{\rm a)}$}~2 av. Adolphe Chauvin (Bat.~E, 5\`eme \'etage), 95302 Cergy-Pontoise, France}
\EmailD{\href{mailto:ruslmax@gmail.com}{ruslmax@gmail.com}, \href{mailto:ruslan.maksimau@cyu.fr}{ruslan.maksimau@cyu.fr}}

\Address{$^{\rm b)}$~Institut de Recherche en Math{\'e}matique et Physique, Universit{\'e} Catholique de Louvain,\\
\hphantom{$^{\rm b)}$}~Chemin du Cyclotron 2, 1348 Louvain-la-Neuve, Belgium}
\EmailD{\href{mailto:pedro.vaz@uclouvain.be}{pedro.vaz@uclouvain.be}}
\URLaddressD{\url{https://perso.uclouvain.be/pedro.vaz}}

\ArticleDates{Received March 30, 2023, in final form November 15, 2023; Published online November 22, 2023}

\Abstract{We construct DG-enhanced versions of the degenerate affine Hecke algebra and of the affine Hecke algebra. We extend Brundan--Kleshchev and Rouquier's isomorphism and prove that after completion DG-enhanced versions of affine Hecke algebras (degenerate or nondegenerate) are isomorphic to completed DG-enhanced versions of KLR algebras for suitably defined quivers. As a byproduct, we deduce that these DG-algebras have homologies concentrated in degree zero. These homologies are isomorphic respectively to the degenerate cyclotomic Hecke algebra and the cyclotomic Hecke algebra.}


\Classification{20C08; 16E45}

\section{Introduction}

Hecke algebras and their affine versions are fundamental objects in mathematics and have a~rich representation theory
 (see, for example, the review~\cite{Kleshchev-review}).
The representation theory of finite dimensional Hecke algebras also carries interesting symmetries which occur
in categorification of Fock spaces and Heisenberg algebras~\cite{Kh-Heisenberg,LicataSavage}.

In a series of outstanding papers, Lauda~\cite{Lauda}, Khovanov--Lauda~\cite{KL1,KL3,KL2}
and independently Rouquier~\cite{R1}, have constructed categorifications of quantum groups.
They take the form of 2-categories whose Grothendieck groups are isomorphic to the
idempotent version of the quantum enveloping algebra of a Kac--Moody algebra.
Both constructions were later proved to be equivalent by Brundan~\cite{Brundan-KLR}.
As a main ingredient of the constructions of Khovanov--Lauda and Rouquier there is a certain family
of algebras, nowadays known as KLR algebras, that are constructed using actions of symmetric groups on polynomial spaces.

It turns out that in type $A$ the KLR algebras are closely related to affine Hecke algebras.
It was proved by Rouquier~\cite[Section 3.2]{R1} that KLR algebras of type $A$ become isomorphic to
affine Hecke algebras after a suitable localization of both algebras.
Independently, Brundan and Klesh\-chev \cite{BK} have proved a similar result for cyclotomic quotient algebras.
This endows cyclotomic Hecke algebras with a presentation as graded idempotented algebras.
In particular, in the case of KLR for the quiver of type $A_\infty$,
the isomorphism to the group algebra of the symmetric group in $d$ letters $\Bbbk\Sy_d$
gives the latter a graded presentation.
The grading on $\Bbbk\Sy_d$ was already known to exist (see~\cite{R2}) but transporting the grading from the
KLR algebras allowed to construct it explicitly.
This gave rise to a new approach to the representation theory of
symmetric groups and Hecke algebras~\cite{BK2}.
These results are valid over an arbitrary field $\Bbbk$.

The BKR (Brundan--Kleshchev--Rouquier) isomorphism was later extended to isomorphisms between families of other KLR-like algebras and Hecke-like algebras. A similar isomorphism between the Dipper--James--Mathas cyclotomic $q$-Schur algebra and the cyclotomic quiver Schur algebra is given in \cite{StroppelWebster}. The authors of~\cite{MaksimauStroppel} and~\cite{webster-Hecke} have constructed a higher level version of the affine Hecke algebra and have proved that after completion they are isomorphic to a~completion of Webster's tensor product algebras~\cite{webster-tensor}. A weighted version of this isomorphism is also given in~\cite{webster-Hecke}. A similar relation between quiver Schur algebras and affine Schur algebras is given in~\cite{MiemietzStroppel}. Also in~\cite{MaksimauStroppel} the authors have constructed a higher level version of the affine Schur algebra and have proved that after completion it is isomorphic to a completion of the higher level quiver Schur algebras.

The BKR isomorphism was also generalized to other algebras. For example,
in~\cite{Rostam} it is used to show that cyclotomic Yokonuma--Hecke algebras are particular
cases of cyclotomic KLR algebras for certain
cyclic quivers, and in~\cite{dannecy-walker} the BKR isomorphism is extended to connect affine Hecke algebras of
type~$B$ and a generalization of KLR algebras for a Weyl group of type~$B$.

Motivated by the work of Khovanov--Lauda~\cite{KL1,KL2}, Rouquier~\cite{R1},
and Kang--Kashiwara~\cite{KangKashiwara}, the second author and Naisse introduced in~\cite{naissevaz3} a family of KLR-like DG-algebras.
{These are referred to as DG-enhanced KLR algebras'' because they are obtained from free resolutions of cyclotomic KLR algebras
over (non-cyclotomic) KLR algebras, where the cyclotomic condition is in some sense replaced by a differential.}
The algebras underlying these DG-algebras also provide categorification of universal Verma modules.

It seems natural to ask the following questions.
\begin{Questions}\quad
\begin{itemize}\itemsep=0pt
 \item[$(a)$] Are there DG-enhanced versions of affine Hecke algebras that are {free resolutions of cyclotomic Hecke algebras over affine Hecke algebras?}
 \item[$(b)$] In this case, does the BKR isomorphism extend to an isomorphism between (completions of)
 DG-enhanced versions of KLR algebras and DG-enhanced versions of Hecke algebras?
 \end{itemize}
\end{Questions}

In this article, we answer these questions affirmatively.

\begin{rem}
In this paper, we work with two versions of affine Hecke algebras, usual affine Hecke algebra, which is an affinization of the Hecke algebra for the symmetric group, and its degenerate version.
We slightly simplify the terminology and refer to these algebras as the \emph{$q$-affine Hecke algebra},
and the \emph{degenerate affine Hecke algebra.}
In fact, our ``affine'' always means ``extended affine''.
\end{rem}

Let us give an overview of our Hecke algebras and the main results in this article.
Fix~$d\in\bN$ (where $0\in\bN$) and a field $\Bbbk$ that for simplicity we consider to be algebraically closed.
We consider the $\bZ$-graded algebra $\widebar{\cH}_d$ generated by $T_1,\dots ,T_{d-1}$ and $X_1,\dots ,X_d$ in degree zero and~$\theta$ in degree~1.
The generators $T_1,\dots ,T_{d-1}$ and $X_1,\dots ,X_d$ satisfy the relations of the degenerate affine Hecke algebra $\widebar{H}_d$. The generator $\theta$ commutes with the $X_r$'s and with $T_2,\dots,T_{d-1}$
and satisfies %the relations
$\theta^2=0$ and
$T_1\theta T_1\theta + \theta T_1\theta T_1 = 0$. This implies that the subalgebra of
$\widebar{\cH}_d$ concentrated in degree zero is isomorphic to $\widebar{H}_d$.
For $\bfQ=(Q_1,\dots ,Q_\ell)\in\Bbbk^\ell$, we
introduce a differential $\partial_\bfQ$ %on $\widebar{\cH}_d$
by declaring that it acts as zero on $\widebar{H}_d$ while
\smash{$\partial_\bfQ(\theta) = \prod_{r=1}^\ell(X_1-Q_r)$}.
We denote by \smash{$\widehat{\bar{\cH}}_\bfa$} the completion of the algebra $\widebar{\cH}_d$ at a sequence of ideals depending on $\bfa\in\Bbbk^d$.

In order to make the connection to DG-enhanced versions of KLR algebras we
consider a~quiver $\Gamma$ with a vertex set $I\subseteq\Bbbk$
and with an edge $i\to j$ iff $j+1=i$. We assume that~$Q_r\in I$ for each $r$.
We fix $\bfa\in I^d$ and we set $\nu$
and $\Lambda$ such that $\nu_i$ and $\Lambda_i$ are the multiplicities of $i$ in respectively $\bfa$ and $\bfQ$.
We have $\prod_{r=1}^\ell(X_1-Q_r)=\prod_{i\in I}(X_1-i)^{\Lambda_i}$.
Let $(\cR(\nu),d_\Lambda)$ be the DG-enhanced version of the KLR algebra of type $\Gamma$ with parameters $\nu$ and $\Lambda$ as above and
$\big(\widehat{\cR}(\nu),d_\Lambda\big)$ its completion.

The first main result in this article is a DG-enhanced version of the BKR isomorphism for the
degenerate affine Hecke algebra:

\begin{Th*}
There is an isomorphism of DG-algebras \smash{$\big(\widehat \cR(\nu),d_\Lambda\big)\simeq \big(\widehat {\bar \cH}_\bfa,\partial_\bfQ\big)$}.
\end{Th*}

There is a similar construction for the affine $q$-Hecke algebra, which we do in~Section~\ref{sec:qversion} and Section~\ref{sec:isoqversion}. Fix $q\in\Bbbk$, $q\neq 0,1$ and denote by $({\cH}_d,\partial_\bfQ)$ and by $\big(\widehat {\cH}_\bfa,\partial_\bfQ\big)$ the DG-enhanced version of the affine $q$-Hecke and its completion. The construction of $\cH_d$ also adds a variable $\theta$ in degree 1 that also satisfies $\theta^2=0$ and commutes with all generators but $T_1$ the relation being $ T_1\theta T_1\theta + \theta T_1\theta T_1 =(q-1)\theta T_1 \theta$.

In a nutshell, fix $\bfQ=(Q_1,\ldots,Q_\ell)\in (\Bbbk^\times)^\ell$. We
consider a quiver $\Gamma$ with a vertex set $I\subseteq\Bbbk^\times$
and with an edge $i\to j$ iff $qj=i$.
We assume that $I$ contains $Q_1,\ldots,Q_\ell$ and fix $\bfa\in I^d$. We define $\nu$
and $\Lambda$ in the same way as above.
Let $(\cR(\nu),d_\Lambda)$ be the DG-enhanced version of the KLR algebra of type $\Gamma$ with $\nu$ and $\Lambda$ as above and let
$\big(\widehat{\cR}(\nu),d_\Lambda\big)$ be its completion.
The second main result in this article is the DG-enhanced version of the BKR isomorphism for the affine $q$-Hecke algebra:

\begin{Th**}
There is an isomorphism of DG-algebras $\big(\widehat \cR(\nu),d_\Lambda\big)\simeq \big(\widehat {\cH}_\bfa,\partial_\bfQ\big)$.
\end{Th**}

 The two main results above imply that we have a family of isomorphisms $\widehat \cR(\nu)\simeq \widehat {\cH}_\bfa$ between the underlying algebras parameterized by integral dominant weights.

The DG-enhanced versions of BKR isomorphisms above allow us to compute the homology of the DG-algebras $\bar\cH_d$ and $\cH_d$ in the following way. It is already proved in \cite[Proposition~4.14]{naissevaz3} that the homology of the DG-algebra $(\cR(\nu),d_\Lambda)$ is concentrated in degree $0$ and is isomorphic to the cyclotomic KLR algebra. The most difficult part of this proof is to show that the homology is concentrated in degree zero. The proof of this fact is quite technical and there is no obvious way to rewrite it for Hecke algebras. So we use the following strategy: we deduce the statement for Hecke algebras from the statement for KLR algebras using the DG-enhanced version of the BKR isomorphism.

As a corollary of Theorems~\ref{thm:isoEDHecke} and~\ref{thm:isoEAHecke} and \cite[Proposition 4.14]{naissevaz3}, the DG-algebras $\big(\widebar{\cH}_d,\partial_\bfQ\big)$ and $(\cH_d,\partial_\bfQ)$ are resolutions of the cyclotomic Hecke algebras $\widebar{H}_d^\bfQ$ and $H_d^\bfQ$. These are cyclotomic quotients of the degenerate affine Hecke algebras and of the affine $q$-Hecke algebras, respectively.

\begin{Prop*}
The homology of the DG-algebra $\big(\widebar{\cH}_d,\partial_\bfQ\big)$ is concentrated in degree $0$ and is isomorphic to \smash{$\widebar{H}_d^\bfQ$}.
\end{Prop*}

\begin{Prop**}
The homology of the DG-algebra $(\cH_d,\partial_\bfQ)$ is concentrated in degree $0$ and is isomorphic to \smash{$H_d^\bfQ$}.
\end{Prop**}

To our knowledge, the DG-enhanced versions of Hecke algebras we introduce are new.
We would also like to emphasize the fact that the algebras $\bar{\cH}_d$ and ${\cH}_d$ have triangular decompositions
(see Remarks~\ref{rem:triang-dec-degen} and \ref{rem:triang-dec-q}). This looks like an analogy with the triangular decomposition in the Cherednik algebras, see also Remark \ref{rem:enh-Hecke-vs-DAHA}.

\subsubsection*{Plan of the paper}

In~Section~\ref{sec:EnhancedHecke}, we introduce DG-enhanced versions of the degenerate affine Hecke algebra and of the affine $q$-Hecke algebra and their completions, that will be used in the BKR
isomorphism. The material in this section is new.

In~Section~\ref{sec:KLRG}, we review the DG-enhanced version of the KLR algebra introduced in~\cite{naissevaz3}.
We give the presentation of this algebra as in~\cite[Corollary 3.16]{naissevaz3}
which is more convenient to us, and present its completion, which is involved in the BKR isomorphism.

Section~\ref{sec:isos} contains the main results.
We first generalize the BKR isomorphism to a class of algebras satisfying some properties.
The most important point is that to have a generalization of the BKR isomorphism we need an isomorphism between the completed polynomial representation of the Hecke-like algebra and the completed polynomial representation of the KLR-like algebra, and this isomorphism must intertwine the action of the symmetric group.
Our main results, Theorems~\ref{thm:isoEDHecke} and~\ref{thm:isoEAHecke}, are then proved by showing
that our DG-enhanced versions of Hecke algebras $\bar\cH_d$ and $\cH_d$ on one side, and
the DG-enhanced versions of KLR algebras~$\cR(\nu)$ on the other side satisfy the properties that are required for them to be isomorphic (after completion).
We then use the DG-enhanced version of the BKR isomorphism and the fact that the DG-algebra~$\cR(\nu)$ is a free resolution of the cyclotomic KLR algebra to show in Corollary~\ref{cor:formal-Hecke} that the algebras $\bar\cH_d$ and $\cH_d$ are free resolutions of the corresponding cyclotomic Hecke algebras.

%%%%%%%%%%%%%%%%%%%%%%%%%%%%%%%%%%%%%%%%%%%%%%%%%%%%
%
% Enhanced Hecke algebras
%
%%%%%%%%%%%%%%%%%%%%%%%%%%%%%%%%%%%%%%%%%%%%%%%%%%%%%
\section{DG-enhanced versions of Hecke algebras}\label{sec:EnhancedHecke}

For integers $a$ and $b$ such that $a\leqslant b$ we write $[a;b]=\{a,a+1,\ldots,b-1,b\}$.

\subsection[The polynomial rings Pol\_d and Poll\_d and the rings P\_d and Pl\_d]{The polynomial rings $\boldsymbol{\Pol_d}$ and $\boldsymbol{\Poll_d}$ and the rings $\boldsymbol{P_d}$ and $\boldsymbol{Pl_d}$}\label{ssec:polyrings}
Fix an algebraically closed field $\Bbbk$, $q\in\Bbbk$, $q\ne 0,1$ and $d\in\bN$ once and for all.

\subsubsection[The polynomial rings Pol\_d and Poll\_d]{The polynomial rings $\boldsymbol{\Pol_d}$ and $\boldsymbol{\Poll_d}$}

Set $\Pol_d=\Bbbk[X_1,\ldots,X_d]$.
Let $\Sy_d$ be the symmetric group on $d$ letters, which we view as a Coxeter group
with generators $s_1,\dots ,s_{d-1}$. These correspond to the simple transpositions~$(i\ i{+}1)$, and we use these two descriptions interchangeably throughout. As usual, we let~$\Sy_d$ act from the left on $\Pol_d$ by permuting the variables:
for $w\in\Sy_d$ we have
\( w(X_i) = X_{w(i)} \),
and \( w(fg) = w(f)w(g) \) for $f,g\in\Pol_d$.

Using the $\Sy_d$-action above, one defines the \emph{Demazure operators} $\partial_i$ on $P_d$
for all $1\leq i \leq d-1$ in the usual way, as
\begin{gather}\label{eq:Demazure}
 \partial_i(f) = \frac{f-s_i(f)}{X_i - X_{i+1}}.
\end{gather}

We have
$s_i\partial_i(f)=\partial_i(f)$ and $\partial_i(s_if)=-\partial_i(f)$ for all $i$,
so $\partial_i$ is in fact an operator from $\Pol_d$ to the subring $\Pol_d^{s_i}\subseteq \Pol_d$
of invariants under the transposition $(i\ i{+}1)$.
It is well known that the action of the Demazure operators on $\Pol_d$ satisfy the Leibniz rule
\begin{gather} \label{eq:D-Leibniz}
\partial_i(fg)=\partial_i(f)g+s_i(f)\partial_i(g),
\end{gather}
for all $f,g\in \Pol_d$ and for $1\leq i\leq n-1$, and the relations
\begin{gather}
\partial_i^2 = 0,\qquad
\partial_i\partial_{i+1}\partial_i =
\partial_{i+1}\partial_i\partial_{i+1} ,\label{eq:DemazureR23}
\\
\partial_i\partial_j =\partial_j\partial_i \qquad \text{for }\quad |i-j| > 1,\label{eq:DemazureComm}
\\
X_i\partial_i - \partial_iX_{i+1} = 1,
\qquad
\partial_iX_{i} - X_{i+1}\partial_i = 1.\label{eq:DemazureX}
\end{gather}

Set $\Poll_d=\Bbbk\big[X_1^{\pm 1},\ldots,X_d^{\pm 1}\big]$, which is the localization of $\Pol_d$ obtained by adding the inverses of $X_1,\ldots,X_d$.
Moreover, the $\Sy_d$-action on $\Pol_d$ can be obviously extended to a $\Sy_d$-action on~$\Poll_d$.
This means that the action of the Demazure operators on $\Pol_d$ also extends to operators on
$\Poll_d$ that satisfy the
relations in~\eqref{eq:D-Leibniz} (for $f$ and $g$ in $\Poll_d$)
and~\eqref{eq:DemazureR23}--\eqref{eq:DemazureX}.

\subsubsection[The rings P\_d and Pl\_d]{The rings $\boldsymbol{P_d}$ and $\boldsymbol{Pl_d}$}

Let $\und{\theta}=\{ \theta_1,\dots,\theta_d \}$ be odd variables and
form the supercommutative ring
\[ P_d=\Pol_d \otimes \bV^\bullet(\und{\theta}), \]
where
$\bV^\bullet(\und{\theta})$ is the exterior $\Bbbk$-algebra in the variables $\und{\theta}$.
Here $P_d$ is a subring concentrated in parity zero.
Introduce an additional $\bZ$-grading on $P_d$ denoted $\lambda(\bullet)$ and defined as $\lambda(X_i)=0$ and~$\lambda(\theta_i)=1$. This grading is half the grading $\deg_\lambda$ introduced in~\cite[Section~3.1]{naissevaz1}. If we forget the grading, the algebra $P_d$ is the symmetric algebra corresponding to a superspace of dimension~$(d\vert d)$.

As explained in~\cite[Section~8.3]{naissevaz1}, the action of $\Sy_d$ on $\Pol_d$ extends to an
action on $P_d$ by setting
\begin{equation}\label{eq:SnActsonR}
s_i(\theta_j) = \theta_{j} + \delta_{i,j}(X_i-X_{i+1})\theta_{i+1} .
\end{equation}
{This action respects the grading, as one easily checks, and allows extending the action of the Demazure operators in~\eqref{eq:Demazure} to $P_d$. We denote the extensions of the Demazure operators to $P_d$ by the same symbols.}
Similarly to the operators above, $\partial_i$ is an operator from $P_d$ to the subring $P_d^{s_i}\subseteq P_d$
of invariants under the transposition $(i\ i{+}1)$.
It was proved in~\cite[Lemma 2.2]{naissevaz2} that the Demazure operators
on $P_d$ satisfy the Leibniz rule~\eqref{eq:D-Leibniz} (for $f,g\in P_d$), the relations~\eqref{eq:DemazureR23}--\eqref{eq:DemazureX} and the following relations:
\begin{gather*}
\partial_i\theta_k = \theta_k\partial_i\qquad\text{for }\quad k\neq i,
 \\
 \partial_i (\theta_i- X_{i+1}\theta_{i+1}) = (\theta_i - X_{i+1}\theta_{i+1})\partial_i,
\end{gather*}
for all $i=1,\dots, d-1$.

As in the case of $P_d$ above, we form the supercommutative ring
\[ Pl_d=\Poll_d \otimes \bV^\bullet(\und{\theta}). \]
This ring is also endowed with the grading $\lambda(\bullet)$, which is defined in the same way as in $P_d$.
Moreover, the $\Sy_d$-action on $\Poll_d$ can be obviously extended to a $\Sy_d$-action on $Pl_d$.
This means that the action of the Demazure operators on $\Poll_d$ also extends to operators
on $Pl_d$
that satisfy the
relations in~\eqref{eq:D-Leibniz} (for $f$ and $g$ in $Pl_d$) and~\eqref{eq:DemazureR23}--\eqref{eq:DemazureX}.

%%%%%%%%%%%%%%%%%%%%%%%%%%%%%%%%%%%%%%%%%%%%%%
%%%
%%% Enhanced Degenerate Affine Hecke
%%%
%%%%%%%%%%%%%%%%%%%%%%%%%%%%%%%%%%%%%%%%%%%%%%
\subsection{Degenerate version}

%%%%%%%%%%%%%%%%%%%%%%%%%%%%%%%%%%%%%%%%%%%%%%
\subsubsection{Degenerate affine Hecke algebra}\label{ssec:DAH}
The \emph{degenerate affine Hecke} algebra $\widebar{H}_d$ is the $\Bbbk$-algebra generated by
$T_1,\dots,T_{d-1}$ and $X_1,\dots,\allowbreak X_d$, with relations %~\eqref{eq:degHeckeSnrels} to~\eqref{eq:degHeckeXTrels} below.
\begin{gather}
 T_i^2 = 1, \qquad T_iT_j =T_jT_i\quad\text{if}\quad
|i-j|>1,\qquad T_iT_{i+1}T_i = T_{i+1}T_iT_{i+1},
 \label{eq:degHeckeSnrels} \\
X_iX_j = X_jX_i,
 \\
T_iX_{i}-X_{i+1}T_i = -1,\qquad
 T_iX_j = X_jT_i\quad \text{if} \quad j-i\neq 0,1.
 \label{eq:degHeckeXTrels}
 \end{gather}

For $w=s_{i_1}\dots s_{i_k}\in\Sy_d$ a reduced expression, we put $T_w=T_{i_1}\dots T_{i_k}$.
Then $T_w$ is independent of the choice of the reduced expression of $w$ and the
set
\[
\big\{ X_1^{m_1}\dots X_d^{m_d}T_w\big\}_{w\in\Sy_d, m_i\in\bZ_{\geq 0}}
\]
is a basis of the $\Bbbk$-vector space $\widebar{H}_d$.

There is a faithful representation of $\widebar{H}_d$ on $\Pol_d$, where
$T_i(f)=s_i(f)-\partial_i(f)$ and $X_i\in\widebar{H}_d$ acts as multiplication by $X_i$.
It is immediate that $\widebar{H}_d$ contains
$\Bbbk\Sy_d$ and $\Pol_d$ as subalgebras and that for $p\in\Pol_d$,
\[
T_i p - s_i(p)T_i = -\partial_i(p) .
\]

Let $\ell$ be a positive integer and $\bfQ=(Q_1,\dots,Q_\ell)$ be an $\ell$-tuple of elements of the field $\Bbbk$.

\begin{defn}\label{def:cyclDAH}
 The \emph{degenerate cyclotomic Hecke algebra} is the quotient
 \[
 \bar{H}_d^\bfQ = \bar{H}_d/\prod\limits_{r=1}^\ell(X_1-Q_r).
 \]
 \end{defn}

\subsubsection[The algebra bar H\_d]{The algebra $\boldsymbol{\bar\cH_d}$}

\begin{defn}\label{def:enhacedDegH}
Define the algebra $\bar{\cH}_d$ as
the $\Bbbk$-algebra generated by
$T_1,\dots ,T_{d-1}$ and $X_1,\dots , \allowbreak X_d$ in $\lambda$-degree zero,
and an extra generator $\theta$ in $\lambda$-degree~1,
with relations~\eqref{eq:degHeckeSnrels} to~\eqref{eq:degHeckeXTrels} and
\begin{gather*}
\theta^2 = 0,
 \\
 X_r\theta = \theta X_r\qquad \text{for} \quad r=1,\dots, d,
 \\
T_r\theta=\theta T_r \qquad \text{for}\quad r>1,
\\
 T_1\theta T_1\theta + \theta T_1\theta T_1 = 0 .
\end{gather*}
\end{defn}

The algebra $\bar{\cH}_d$ contains the degenerate affine Hecke algebra $\bar{H}_d$ as a subalgebra concentrated
in $\lambda$-degree zero.

\begin{lem}
\label{lem:barcH-acts-Pd}
The algebra $\bar{\cH}_d$ acts on $P_d$ by
\begin{align*}
& T_r(f)= s_r(f) - \partial_r (f) ,
\qquad X_r(f)= X_rf ,
 \qquad \theta(f)= \theta_1f ,
\end{align*}
for all $f\in P_d$ and where $s_r(f)$ and $\partial_r(f)$ are as
in~\eqref{eq:SnActsonR} and~\eqref{eq:Demazure}.
\end{lem}
\begin{proof}
 The defining relations of $\bar{\cH}_d$ can be checked by a straightforward computation.
\end{proof}

Define $\xi_1,\dots , \xi_d\in \bar{\cH}_d$ by the rules $\xi_1=\theta$, $\xi_{i+1}=T_i\xi_iT_i$.
The following is straightforward.

\begin{lem}\label{lem:xisDegH}
 The elements $\xi_r$ satisfy for all $r\in\{1,\dots ,d-1\}$ and all $\ell\in\{1,\dots, d\}$,
 \[
\xi_\ell^2=0,\qquad \xi_r\xi_\ell + \xi_\ell\xi_r = 0,\qquad T_r\xi_\ell = \xi_{s_r(\ell)} T_r.
 \]
\end{lem}

\begin{rem}
\label{rem:enh-Hecke-vs-DAHA}
It is easy to give the relations between $T$'s and $X$'s and between $T$'s and $\xi$'s. However, $X$'s and $\xi$'s satisfy more elaborate relations, which is similar to what happens with two polynomial rings in Cherednik (double affine Hecke) algebras. For example, the following commutation relations can be checked easily:
\begin{gather*}
[X_r,\xi_1]=0 ,\qquad
[X_1,\xi_2]=-[X_2,\xi_2]=[\xi_2,T_1]=[T_1,\xi_1] ,
\qquad
[X_1,\xi_3]=T_2[T_1,\xi_1]T_2 .
\end{gather*}
\end{rem}

Abusing the notation, we will write $\theta_r$ for the operator on $P_d$ that multiplies each element of~$P_d$ by $\theta_r$.
Set $M=\{0,1\}^d$. Denote by $\mathbf{1}$ the sequence $\mathbf{1}=(1,1,\ldots,1)\in M$.
For each sequence~$\bfb=(b_1,\ldots,b_d)\in M$, we set $\theta^\bfb=\theta_1^{b_1}\dots \theta_d^{b_d}$. For each $\bfb\in M$, we set $\overline\bfb=\mathbf{1}-\bfb$. In particular, we have \smash{$\theta^\bfb\cdot\theta^{\overline \bfb}=\pm \theta_1\theta_2\dots \theta_d=\pm\theta^{\mathbf 1}$}. Set also $|\bfb|=b_1+b_2+\dots+b_d$.

\begin{lem}
\label{lem:theta-free}
The operators $\big\{\theta^\bfb\mid \bfb\in M\big\}$ acting on $P_d$ are linearly independent over $\bar{H}_d$. More precisely, if we have $\sum_{\bfb\in M}h_\bfb\theta^\bfb=0$ with $h_\bfb\in \bar{H}_d$, then we have $h_\bfb=0$ for each $\bfb\in M$.
\end{lem}
\begin{proof}
Let $H=\sum_{\bfb\in M}h_\bfb\theta^\bfb$ be an operator that acts by zero. Assume that $H$ has a nonzero coefficient. Let $\bfb_0$ be such that $h_{\bfb_0}\ne 0$ and such that $|\bfb_0|$ is minimal with this property. Then for each element $P\in P_d$, we have \smash{$H\big(\theta^{\overline{\bfb_0}}P\big)=\pm h_{\bfb_0}\theta^{\mathbf 1}P$}. This shows that $h_{\bfb_0}$ acts by zero on~$\theta^{\mathbf 1}P_d=\theta^{\mathbf 1}\Pol_d$. But this implies $h_{\bfb_0}=0$ because the polynomial representation of $\bar{H}_d$ on~$\Pol_d$ is faithful, see \cite[Section 3.1.2]{R1}.
\end{proof}

For each, $k\in\{0,1,\ldots,d\}$ we denote by \smash{$\bar{\cH}^{\leqslant k}_d$} the subalgebra of the algebra of operators on~$P_d$ generated by $X_i$, $\theta_i$ for $i\leqslant k$ and $T_r$ for $r<k$. Denote also by \smash{$\bar{H}^{\leqslant k}_d$} the subalgebra of $\bar{H}_d$ generated by $X_i$ for $i\leqslant k$ and $T_r$ for $r<k$. Since $\bar{H}_d$ acts faithfully on $P_d$, we can see \smash{$\bar{H}^{\leqslant k}_d$} as a subalgebra of \smash{$\bar{\cH}^{\leqslant k}_d$}. We mean that for $k=0$ we have \smash{$\bar{\cH}^{\leqslant 0}_d=\bar{H}^{\leqslant 0}_d=\Bbbk$}. The $\lambda$-grading on $P_d$ induces a grading on \smash{$\bar{\cH}^{\leqslant k}_d$} that we also call $\lambda$-grading.
\begin{lem}
\label{lem:basis-theta-trunc}
The set
\[
 \big\{ X_1^{a_1}\dots X_k^{a_k}T_w \theta_1^{b_1}\dots \theta_k^{b_k} \mid w\in\Sy_k,\, (a_1,\dots, a_k) \in\bN^k , \, (b_1,\dots, b_k)
 \in\{0,1\}^k \big\} ,
\]
is a basis of the $\Bbbk$-vector space $\bar{\cH}^{\leqslant k}_d$.
\end{lem}
\begin{proof}
It is clear that the given set spans. Linear independence follows from Lemma~\ref{lem:theta-free}.
\end{proof}

Similarly to the notation $\theta^\bfb$ above, we set $\xi^\bfb=\xi_1^{b_1}\dots \xi_d^{b_d}$. For two elements $\bfb,\bfb'\in M$, we write $\bfb'<\bfb$ if there is an index $r\in[1;d]$ such that $b'_r<b_r$ and $b'_{t}= b_t$ for $t>r$.
For $\bfb\in M$, write $\max(\bfb)$ for the maximal index $r\in [1;d]$ such that $b_r=1$.

\begin{lem}
\label{lem:xi-k-acts}
The element $\xi_k$ acts on $P_d$ by an operator of the form $c_k+d_k\theta_k$, where $c_k\in \bar{\cH}^{\leqslant k-1}_d$, \smash{$d_k\in \bar{H}^{\leqslant k-1}_d$}, $\lambda(c_k)=1$ and $d_k$ is not a right zero divisor in $\bar{H}_d$.
\end{lem}
\begin{proof}
We prove by induction on $k$. The case $k=1$ is trivial.
Now, assume that $d_k$ is not a~right zero divisor and let us show that $d_{k+1}$ is not a~right zero divisor. Since we have
\begin{align*}
T_{k}d_{k}\theta_{k}T_{k} &= T_{k}d_{k}T_{k}(\theta_{k}+(X_{k}-X_{k+1})\theta_{k+1})+T_{k}d_{k}\theta_{k+1}
\\
&= T_{k}d_{k}T_{k}\theta_{k}+(T_{k}d_{k}T_{k}(X_{k}-X_{k+1})+T_{k}d_{k})\theta_{k+1},
\end{align*}
we get
\begin{align*}
 d_{k+1} &= T_{k}d_{k}T_{k}(X_{k}-X_{k+1})+T_{k}d_{k} =T_{k}d_{k}(T_{k}(X_{k}-X_{k+1})+1) \\ &= T_{k}d_{k}((X_{k+1}-X_{k})T_{k}-1).
\end{align*}
It is enough to check that the element $((X_{k+1}-X_{k})T_{k}-1)$ is not a right zero divisor. This follows from the fact that it acts on $P_d$ by the operator $(X_{k+1}-X_{k}-1)s_k$.
\end{proof}

\begin{lem}
\label{lem:xi-b-acts}
The element $\xi^\bfb\in\bar{\cH}_d$ acts on $P_d$ by an operator of the form $c_\bfb+d_{\bfb}\theta^{\bfb}$, where \smash{$d_\bfb\in \bar{H}_d^{\leqslant\max(\bfb)-1}$} and $d_\bfb$ is not a right zero divisor in $\bar{H}_d$ and $c_{\bfb}$ is of the form $\sum_{\bfb'<\bfb}h_{\bfb'}\theta^{\bfb'}$ with \smash{$h_{\bfb'}\in \bar{H}_d^{\leqslant\max(\bfb)-1}$}.
\end{lem}
\begin{proof}
We prove the statement by induction on $|\bfb|=r$. The case $r=1$ follows immediately from the lemma above. Now, for $r>1$, assume that the statement is true for $r-1$, let us prove it for $r$.

Set $p=\max(\bfb)$. Let $\bfb_1\in M$ be such that $\theta^\bfb=\theta^{\bfb_1}\theta_{p}$. By the induction assumption, the element $\xi^\bfb=\xi^{\bfb_1}\xi_{p}$ acts on $P_d$ by an operator of the form (up to sign) $(c_p+d_p\theta_p)\big(c_{\bfb_1}+d_{\bfb_1}\theta^{\bfb_1}\big)$. This operator can be written as $c_\bfb+d_\bfb\theta^\bfb$ for $d_\bfb=d_pd_{\bfb_1}$ and $c_\bfb=c_p(c_{\bfb_1}+d_{\bfb_1}\theta^{\bfb_1})+d_p\theta_pc_{\bfb_1}$. Now, we obviously get $d_\bfb\in \bar{H}_d^{\leqslant p-1}$ because it is a product of two elements of $\bar{H}_d^{\leqslant p-1}$ and it is not a right zero divisor as a product of two right non-zero divisors. Moreover, the element~$c_\bfb$ is of the form $\sum_{\bfb'<\bfb}h_{\bfb'}\theta^{\bfb'}$ because $d_p\theta_pc_{\bfb_1}=d_pc_{\bfb_1}\theta_p$ is of the required form and because~\smash{$c_p(c_{\bfb_1}+d_{\bfb_1}\theta^{\bfb_1})\in \bar{\cH}_d^{\leqslant p-1}$} (and then it is also of the required form).
\end{proof}

It is not hard to write a basis of $\bar{\cH}_d$ in terms of the $\xi_r$'s.
\begin{prop} \label{prop:basis-daH}
The set
\begin{gather*}
 \big\{ X_1^{a_1}\dots X_d^{a_d}T_w \xi_1^{b_1}\dots \xi_d^{b_d} \mid w\in\Sy_d,\, (a_1,\dots, a_d) \in\bN^d , \, (b_1,\dots, b_d)
 \in\{0,1\}^d \big\} ,
\end{gather*}
is a basis of the $\Bbbk$-vector space $\bar{\cH}_d$.
\end{prop}
\begin{proof}
We start by showing that this set spans $\bar{\cH}_d$. First, each monomial on $\theta$, $X$'s and $T$'s can be rewritten as a linear combination of similar monomials with all $X$'s on the left. After that, we replace $\theta$ by $\xi_1$ and we move all $\xi$'s to the right by using~Lemma~\ref{lem:xisDegH}. This shows that the set above spans $\bar{\cH}_d$.
Linear independence follows from Lemmas~\ref{lem:theta-free} and~\ref{lem:xi-b-acts}.
\end{proof}

\begin{cor}
The representation defined in Lemma~{\rm \ref{lem:barcH-acts-Pd}} is faithful.
\end{cor}
\begin{proof}
We see from the proof of the proposition above that the elements of the basis act by linearly independent operators.
\end{proof}

\begin{rem}
\label{rem:triang-dec-degen}
We see from~Proposition~\ref{prop:basis-daH} that the algebra $\bar{\cH}_d$ has a triangular decomposition (only as a vector space)
\[
\bar{\cH}_d\cong \Bbbk[X_1,\ldots,X_d]\otimes \Bbbk\Sy_d\otimes \bV^\bullet(\xi_1,\ldots,\xi_d).
\]
\end{rem}

\subsubsection[DG-enhancement of bar H\_d]{DG-enhancement of $\boldsymbol{\bar\cH_d}$}

Let $\ell$ and $\bfQ$ be as in~Section~\ref{ssec:DAH}.

\begin{defn}
Define an operator $\partial_\bfQ$ on $\widebar{\cH}_d$
by declaring that
 $\partial_\bfQ$ acts as zero on $\widebar{H}_d\subseteq \widebar{\cH}_d$, while
 \[
 \partial_\bfQ(\theta) = \prod_{r=1}^\ell(X_1-Q_r) ,
 \]
 and it respects the graded Leibniz rule: for $a,b\in\widebar{\cH}_d$, $\partial_\bfQ(ab)=\partial_\bfQ(a)b+(-1)^{\lambda(a)}a\partial_\bfQ(b)$.
\end{defn}

\begin{lem} \label{lem:diff-daH}
The operator $\partial_\bfQ$ is a differential on $\widebar{\cH}_d$.
\end{lem}
\begin{proof}
 We prove something slightly more general.
 Let $P\in\Bbbk[X_1]$ be a polynomial. Define $d_P\colon\widebar{\cH}_d\to \widebar{\cH}_d$ by declaring that
$d_P$ acts as zero on $\widebar{H}_d$, while
 $ d_P(\theta) = P$, together with the graded Leibniz rule.
 Then $d_P$ is a differential on $\widebar{\cH}_d$. To prove the claim is suffices to
 check that~${d_P( T_1\theta T_1\theta + \theta T_1\theta T_1) =0}$.

We have $T_1P=s_1(P)T_1-\partial_1(P)$ and $PT_1=T_1s_1(P)-\partial_1(P)$, where $\partial_1$ is the Demazure operator. This also implies $T_1PT_1=s_1(P)-\partial_1(P)T_1$. Note also that $\partial_1(P)$ is a symmetric polynomial with respect to $X_1,X_2$, so it commutes with $T_1$.
So, we have
 \begin{align*}
 d_P(T_1\theta T_1\theta + \theta T_1\theta T_1)
 ={}& T_1PT_1\theta-T_1\theta T_1P+PT_1\theta T_1-\theta T_1PT_1\\
 ={}& (s_1(P)\theta-\partial_1(P)T_1\theta)-(T_1\theta s_1(P)T_1-T_1\theta \partial_1(P))\\
 & + (T_1s_1(P)\theta T_1-\partial_1(P)\theta T_1)-(\theta s_1(P)-\theta \partial_1(P)T_1) = 0 ,
 \end{align*}
which proves the claim.
\end{proof}

We will prove in Proposition~\ref{prop:formalDHecke} that the homology of the DG-algebra $\big(\widebar{\cH}_d,\partial_\bfQ\big)$ is concentrated in degree $0$ and is isomorphic to $\widebar{H}_d^\bfQ$.

\subsubsection[Completions of bar H\_d]{Completions of $\boldsymbol{\bar\cH_d}$}
\label{ssec:compl_deg-Hecke-Gr}

Consider the algebra of symmetric polynomials $\Sym_d=\Pol_d^{\Sy_d}$.
We consider it as a (central) subalgebra of $\bar{\cH}_d$.

For each $d$-tuple $\bfa=(a_1,\ldots,a_d)\in\Bbbk^d$ we have a character $\chi_\bfa\colon\Sym_d\to \Bbbk$ given by the evaluation $X_r\mapsto a_r$. It is obvious from the definition that if the $d$-tuple $\bfa'$ is a permutation of the $d$-tuple $\bfa$ then the characters $\chi_\bfa$ and $\chi_{\bfa'}$ are the same.
Denote by $\frakm_\bfa$ the kernel of $\chi_\bfa$.

\begin{defn}
Denote by $\widehat{\bar{\cH}}_\bfa$ the completion of the algebra $\bar{\cH}_d$ with respect to $\frakm_\bfa$.
\end{defn}

Since $\frakm_\bfa$ is in the kernel of $\partial_\bfQ$, we can extend $\partial_\bfQ$ to $\widehat{\bar{\cH}}_\bfa$.
Set also
\begin{gather*}
\widehat\Pol_\bfa=\bigoplus_{\bfb\in\Sy_d\bfa}\Bbbk[[X_1-b_1,\ldots,X_d-b_d]]1_\bfb,\\ \widehat P_\bfa=\bigoplus_{\bfb\in\Sy_d\bfa}(\Bbbk[[X_1-b_1,\ldots,X_d-b_d]]\otimes \bV^\bullet(\und{\theta}))1_\bfb,
\end{gather*}
where $1_\bfb$ is just a formal idempotent projecting on the corresponding direct factor.
Here $\Sy_d\bfa$ is the $\Sy_d$-orbit of $\bfa$ with respect to the obvious $\Sy_d$-action on $\Bbbk^d$.
We can obviously extend the action of $\bar{\cH}_d$ on $P_d$ to an action of \smash{$\widehat{\bar{\cH}}_\bfa$} on \smash{$\widehat P_\bfa$}.
Each finite dimensional \smash{$\widehat{\bar{\cH}}_\bfa$}-module $M$ decomposes into its generalized eigenspaces $M=\bigoplus_{\bfb\in\Sy_d\bfa}M_\bfb$, where
\[
M_\bfb=\big\{m\in M\mid \exists N\in\mathbb{N} \mbox{ such that } (X_r-b_r)^Nm=0~\forall r\big\}.
\]
For each $\bfb\in\Sy_d\bfa$ the algebra $\widehat{\bar{\cH}}_\bfa$ contains an idempotent $1_\bfb$ that projects onto $M_\bfb$ when applied to $M$.

\begin{prop}
\label{prop:faithrep+basis-degen-Gr} \quad
\begin{itemize}\itemsep=0pt
 \item[$(a)$] The $\widehat{\Pol}_\bfa$-module \smash{$\widehat{\bar{\cH}}_\bfa$} is free with basis
 \[ \big\{ T_w\xi_1^{b_1}\ldots\xi_d^{b_d} \mid w\in\Sy_d,\, (b_1,\ldots,b_d)\in\{0,1\}^d\big\}. \]
 \item[$(b)$] The representation $\widehat{P}_\bfa$ of \smash{$\widehat{\bar{\cH}}_\bfa$} is faithful.
\end{itemize}
\end{prop}
\begin{proof}
 It is clear that the elements from the statement generate the $\widehat{\Pol}_\bfa$-module
 \smash{$\widehat{\bar\cH}_\bfa$}. To see that they form a basis, it is enough to remark that they act by linear independent \big(over $\widehat{\Pol}_\bfa$\big) operators on the representation \smash{$\widehat{P}_\bfa$}. This proves $(a)$. Then $(b)$ also holds because a basis acts on \smash{$\widehat{P}_\bfa$} by linearly independent operators.
\end{proof}

The algebra \smash{$\bar{H}^\bfQ_d$} has a decomposition \smash{$\bar{H}^\bfQ_d=\oplus_\bfa\bar{H}^\bfQ_\bfa$} (with a finite number of nonzero terms) such that $\Sym_d$ acts on each finite dimensional \smash{$\bar{H}^\bfQ_\bfa$}-module with a generalized character $\chi_\bfa$.

\subsection[q-version]{$\boldsymbol{q}$-version}\label{sec:qversion}

\subsubsection[Affine q-Hecke algebra]{Affine $\boldsymbol{q}$-Hecke algebra}\label{ssec:AH}
The \emph{affine $q$-Hecke algebra} $H_d$ is the $\Bbbk$-algebra generated by
$T_1,\dots ,T_{d-1}$ and $X_1^{\pm 1},\dots ,X_d^{\pm 1}$,
with relations %~\eqref{eq:affHeckeX}-\eqref{eq:affHeckeXT} below.
\begin{gather}\label{eq:affHeckeX}
X_rX_r^{-1}=X_r^{-1}X_r=1, \qquad
X_i^{\pm 1}X_j^{\pm 1} = X_j^{\pm 1}X_i^{\pm 1},
\\\label{eq:affHeckeT}
(T_i-q)(T_i+1)=0,\qquad T_iT_j=T_jT_i~\text{ if }| i-j| >1 ,
\qquad T_iT_{i+1}T_i = T_{i+1}T_{i}T_{i+1},
\\ \label{eq:affHeckeXT}
T_i X_j = X_j T_i\mspace{20mu}\text{for\ } j-i\ne 0,1,
\quad
T_i X_{i} T_i =qX_{i+1}.
\end{gather}

Note that relation \eqref{eq:affHeckeT} implies that the element $T_i$ is invertible.
For $w=s_{i_1}\dots s_{i_k}\in\Sy_d$ a reduced decomposition, we put $T_w=T_{i_1}\dots T_{i_k}$.
Then $T_w$ is independent of the choice of the reduced decomposition of $w$ and the set
\[
\big\{ X_1^{m_1}\dots X_d^{m_d}T_w\big\}_{w\in\Sy_d, m_i\in\bZ}
\]
is a basis of the $\Bbbk$-vector space $H_d$.
There is a faithful representation of $H_d$ on $\Poll_d$, where
$T_i(f)=qs_i(f)-(q-1)X_{i+1}\partial_i(f)$.

Let $\ell$ be a positive integer. Let $\bfQ=(Q_1,\dots,Q_\ell)$ be an $\ell$-tuple of nonzero elements of the field $\Bbbk$.

\begin{defn}\label{def:cyclAH}
 The \emph{cyclotomic $q$-Hecke algebra} is the quotient
 \[
 H_d^\bfQ = H_d/\prod_{r=1}^\ell(X_1-Q_r).
 \]
 \end{defn}

\subsubsection[The algebra H\_d]{The algebra $\boldsymbol{\cH_d}$}\label{ssec:affGHecke}

\begin{defn}\label{def:enhacedqH}
The algebra $\cH_d$ is the $\Bbbk$-algebra generated by
$T_1,\dots ,T_{d-1}$ and $X_1^{\pm 1},\dots ,X_d^{\pm 1}$ in $\lambda$-degree zero,
and an extra generator $\theta$ in $\lambda$-degree 1,
with relations~\eqref{eq:affHeckeX} to~\eqref{eq:affHeckeXT} and
\begin{gather*}
\theta^2 = 0,\qquad
 X_r^{\pm 1}\theta = \theta X_r^{\pm 1}\quad \text{for}\quad r=1,\dots, d,
 \\
T_r\theta=\theta T_r \quad \text{for}\quad r>1,
\\
 T_1\theta T_1\theta + \theta T_1\theta T_1 = (q-1)\theta T_1\theta .
\end{gather*}
\end{defn}

The algebra $\cH_d$ contains the affine $q$-Hecke algebra $H_d$ as a subalgebra concentrated in $\lambda$-degree zero.

\begin{lem}
\label{lem:polrep-HeckeGrassm-q}
The algebra ${\cH}_d$ acts on $Pl_d$ by
\begin{align*}
&T_r(f) = qs_r(f) - (q-1)X_{r+1}\partial_r(f) , \qquad
 X_r^{\pm 1}(f) = X_r^{\pm 1}f ,
\qquad
 \theta(f) = \theta_1f ,
\end{align*}
for all $f\in P_d$ and where $s_r(f)$ and $\partial_r(f)$ are
as in~\eqref{eq:SnActsonR} and~\eqref{eq:Demazure}.
\end{lem}
\begin{proof}
 The defining relations of $\cH_d$ can be checked by a straightforward computation.
\end{proof}

Define $\xi_1,\dots , \xi_d\in {\cH}_d$ by the rules $\xi_1=\theta$, $\xi_{i+1}=T_i\xi_iT^{-1}_i$.
The following is straightforward.

\begin{lem}\label{lem:xisqH}
 The elements $\xi_r$ satisfy for all $r=1,\dots ,d-1$ and all $\ell=1,\dots, d$,
 \[
\xi_\ell^2=0,\qquad \xi_r\xi_\ell + \xi_\ell\xi_r = 0
\]
and
\[
T_\ell\xi_r =
\begin{cases}
 \xi_rT_\ell & \text{ if }\ r\neq \ell,\ell+1 ,
 \\
 \xi_\ell T_\ell + (q-1)(\xi_{\ell+1}-\xi_\ell) & \text{ if }\ r = \ell+1 ,
 \\
 \xi_{\ell+1} T_\ell & \text{ if }\ r = \ell .
\end{cases}
\]
\end{lem}

It is not hard to write a basis of ${\cH}_d$ in terms of the $\xi_r$'s.

\begin{prop} \label{prop:basis-qAGH}
The set
\begin{gather*}
 \big\{ X_1^{a_1}\dots X_d^{a_d}T_w \xi_1^{b_1}\dots \xi_d^{b_d} \mid w\in\Sy_d,\, (a_1,\dots, a_d) \in\bZ^d ,\, (b_1,\dots, b_d)
 \in\{0,1\}^d \big\} ,
\end{gather*}
is a basis of the $\Bbbk$-vector space ${\cH}_d$.
\end{prop}
\begin{proof}
Imitate the proof of~Proposition~\ref{prop:basis-daH}.
\end{proof}

\begin{cor}
 The representation defined in Lemma~{\rm \ref{lem:polrep-HeckeGrassm-q}} is faithful.
\end{cor}

\begin{rem}
\label{rem:triang-dec-q}
We see from Proposition~\ref{prop:basis-qAGH} that the algebra ${\cH}_d$ has a triangular decomposition (only as a vector space)
\[
{\cH}_d=\Bbbk\big[X_1^{\pm 1},\ldots,X_d^{\pm 1}\big]\otimes {H}^{\rm fin}_d \otimes \bV^\bullet(\xi_1,\ldots,\xi_d),
\]
where $H^{\rm fin}_d$ is the (finite dimensional) Hecke algebra of the group $\Sy_d$. Explicitly, the algebra~$H^{\rm fin}_d$ is defined by generators $T_1,\ldots, T_{d-1}$ and the relations in \eqref{eq:affHeckeT}.
\end{rem}

\subsubsection[DG-enhancement of H\_d]{DG-enhancement of $\boldsymbol{\cH_d}$}

Let $\ell$ and $\bfQ$ be as in~Section~\ref{ssec:AH}.

\begin{defn}
 Define an operator $\partial_\bfQ$ on $\cH_d$
by declaring that
 $\partial_\bfQ$ acts as zero on $H_d$, while%
 \[
 \partial_\bfQ(\theta) = \prod_{r=1}^\ell(X_1-Q_r) ,
 \]
and for $a,b\in\widebar{\cH}_d$, $\partial_\bfQ(ab)=\partial_\bfQ(a)b+(-1)^{\lambda(a)}a\partial_\bfQ(b)$.
\end{defn}

\begin{lem}
The operator $\partial_\bfQ$ is a differential on $\cH_d$.
\end{lem}
\begin{proof}
Similarly to the proof of~Lemma~\ref{lem:diff-daH}, we consider a more general differential $d_P$.
We have to check
\[
d_P(T_1\theta T_1\theta + \theta T_1\theta T_1)= d_P((q-1)\theta T_1\theta).
\]

We have $T_1P=s_1(P)T_1-(q-1)X_2\partial_1(P)$ and $PT_1=T_1s_1(P)-(q-1)X_2\partial_1(P)$, where $\partial_1$ is the Demazure operator. Note also that $\partial_1(P)$ is a symmetric polynomial with respect to $X_1$,~$X_2$, so it commutes with $T_1$.
So, we have
 \begin{align*}
 d_P(T_1\theta T_1\theta + \theta T_1\theta T_1)
 ={}& T_1PT_1\theta-T_1\theta T_1P+PT_1\theta T_1-\theta T_1PT_1\\
 ={}& \big(T_1^2s_1(P)\theta-(q-1)\partial_1(P)T_1X_2\theta\big)-(T_1\theta s_1(P)T_1 \\
 & -(q-1)T_1\theta X_2\partial_1(P)) +(T_1s_1(P)\theta T_1-(q-1)X_2\partial_1(P)\theta T_1) \\
 & -(\theta s_1(P)T_1^2-(q-1)\theta \partial_1(P)X_2T_1) \\
 ={}& T_1^2s_1(P)\theta-\theta s_1(P)T_1^2\\
 ={}& (q-1)PT_1\theta-(q-1)\theta T_1P\\
 ={}& d_P((q-1)\theta T_1\theta) ,
 \end{align*}
which proves the claim.
\end{proof}

We will prove in Proposition~\ref{prop:formalAHecke} that the homology of the DG-algebra $(\cH_d,\partial_\bfQ)$ is concentrated in degree $0$ and is isomorphic to $H_d^\bfQ$.

\subsubsection[Completions of H\_d]{Completions of $\boldsymbol{\cH_d}$}
\label{ssec:compl-affHeckGr}

Similarly to~Section~\ref{ssec:compl_deg-Hecke-Gr}, we want to define a completion of the algebra $\cH_d$.
Consider the algebra of symmetric Laurent polynomials \smash{$\Syml_d=\Bbbk\big[X_1^{\pm 1},\ldots,X_d^{\pm 1}\big]^{\Sy_d}$}.
We consider it as a~(central) subalgebra of $\cH_d$.

For each $d$-tuple $\bfa=(a_1,\ldots,a_d)\in(\Bbbk^\times)^n$, we have a character $\chi_\bfa\colon\Syml_d\to \Bbbk$ given by the evaluation $X_r\mapsto a_r$. Denote by $\frakm_\bfa$ the kernel of $\chi_\bfa$.

\begin{defn}
Denote by $\widehat{\cH}_\bfa$ the completion of the algebra $\cH_d$ at the sequence of ideals $\cH_d\frakm_\bfa^j \cH_d$.
\end{defn}

Since $\frakm_\bfa$ is in the kernel of $\partial_\bfQ$, we can extend $\partial_\bfQ$ to $\widehat{\cH}_\bfa$.
Set also \[\widehat P_\bfa=\Bbbk[[X_1-a_1,\ldots, X_d-a_d]]\otimes \bV^\bullet(\und{\theta}).\]
We can obviously extend the action of $\cH_d$ on $P_d$ to an action of $\widehat{\cH}_\bfa$ on $\widehat P_\bfa$.
Similarly to $\widehat{\bar{\cH}}_\bfa$, the algebra $\widehat{\cH}_\bfa$ has idempotents $1_\bfb$, $\bfb\in \Sy_d\bfa$ that are defined in the same way as in~Section~\ref{ssec:compl_deg-Hecke-Gr}.

Similar to~Proposition~\ref{prop:faithrep+basis-degen-Gr}, we have the following.
\begin{prop}\label{prop:faithrep+basis-affine-Gr}\quad
\begin{itemize}\itemsep=0pt
\item[$(a)$] The $\widehat{\Pol}_\bfa$-module $\widehat{\cH}_\bfa$ is free
 with basis
 \[
 \big\{ T_w\xi_1^{b_1}\ldots\xi_d^{b_d} \mid w\in\Sy_d,\, (b_1,\ldots,b_d)\in \{0,1\}^d\big\}.
 \]
\item[$(b)$] The representation $\widehat{P}_\bfa$ of $\widehat{\cH}_\bfa$ is faithful.
\end{itemize}
\end{prop}

The algebra ${H}^\bfQ_d$ has a decomposition ${H}^\bfQ_d=\oplus_\bfa{H}^\bfQ_\bfa$ (with a finite number of nonzero terms) such that $\Syml_d$ acts on each finite dimensional $H^\bfQ_\bfa$-module with a generalized character $\chi_\bfa$.

\section{DG-enhanced versions of KLR algebras}\label{sec:KLRG}

DG-enhanced versions of KLR algebras were introduced in~\cite{naissevaz3}
as one of the main ingredients in the categorification of Verma modules
for symmetrizable quantum Kac--Moody algebras.

Let $\Gamma=(I,A)$ be a quiver without loops with set of vertices $I$ and set of arrows $A$.
We call elements in $I$ \emph{labels}.
Let also $\bN[I]$ be the set of formal $\bN$-linear combinations of elements of $I$. Fix $\nu \in \bN[I]$,
\[
\nu = \sum_{i\in I} \nu_i \cdot i, \qquad \nu_i \in \bN,\quad i\in I,
\]
and set $|\nu| = \sum_i \nu_i$. We allow the quiver to have infinite number of vertices. In this case, only a~finite number of $\nu_i$ is nonzero.

For each $i, j\in I$, we denote by $h_{i,j}$ the number of arrows in
the quiver $\Gamma$ going from $i$ to $j$, and define for $i \neq j$ the polynomials
\[ \cQ_{i,j}(u,v)=(u-v)^{h_{i,j}}(v-u)^{h_{j,i}} . \]

\subsection[The algebra R(nu)]{The algebra $\boldsymbol{\cR(\nu)}$}\label{subsec:diags}

We give a diagrammatic definition of the algebras $\cR=\cR(\Gamma)$ from~\cite[Section~3]{naissevaz3}.
The definition we give corresponds to the presentation in~\cite[Corollary~3.16]{naissevaz3}.

\begin{defn}\label{def:KLR-enh}
 For each $\nu\in\bN[I]$, we define the $\Bbbk$-algebra $\cR(\nu)$ by the data:
\begin{itemize}\itemsep=0pt
\item It is generated by the \emph{KLR generators}
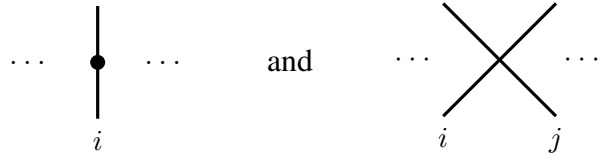
\[
\tikz[very thick,scale=1.5,baseline={([yshift=.8ex]current bounding box.center)}]{
  \draw (0,-.5) node[below] {\small $i$} -- (0,.5) node [midway,fill=black,circle,inner sep=2pt]{};
 	 \node at (-.60,0){$\dots$};
  \node at (.60,0){$\dots$};
} \mspace{40mu}\text{and}\mspace{40mu}
\tikz[very thick,scale=1.5,baseline={([yshift=.8ex]current bounding box.center)}]{
	  \draw (0,-.5) node[below] {\small $i$} -- (1,.5);
	  \draw (1,-.5) node[below] {\small $j$} -- (0,.5);
 	 \node at (-.25,0){$\cdots$};
  \node at (1.25,0){$\cdots$};
	}
 \]
 for $i,j\in I$, where each diagram contains $\nu_i$ strands labeled $i$,
 together with \emph{floating dots} that are confined to a region immediately to the right
 of the left-most strand,
\[
\tikz[very thick,baseline={([yshift=-.5ex]current bounding box.center)},decoration={
 markings, mark=at position 0.5 with {\arrow{>}}}] {
 \draw (-1,-.5) node[below] {\small $i$} -- (-1,.5);
 \node at (.5,0){$\cdots$};
	 \fdot[]{}{-.5,0};
	 	}.
 \]
Diagrams are taken modulo isotopies that do not allow triple crossings of strands, do not allow a dot going through a crossing, and do not allow two floating dots at the same level.

\item The multiplication is given by gluing diagrams on top of each other\footnote{We follow the usual (and useful) convention that $m$ dots on the same strand are depicted as a single dot with an exponent $m$.} whenever the labels of the strands agree, and zero otherwise, subject to the local relations~\eqref{eq:KLRR2} to~\eqref{eq:extrarel} below,
 for all $i,j,k \in I$.
\begin{itemize}\itemsep=0pt
\item[$\diamond$] The \emph{KLR relations, for all $i, j, k\in I$}:
\begin{gather}\label{eq:KLRR2}
	\tikz[very thick,baseline={([yshift=1.0ex]current bounding box.center)}]{
	 \draw +(0,-.75) node[below] {\small $i$} .. controls (1,0) .. +(0,.75);
	 \draw +(1,-.75) node[below] {\small $i$} .. controls (0,0) .. +(1,.75);
	} = 0 \qquad\text{and}\qquad
 \tikz[very thick,baseline={([yshift=1.0ex]current bounding box.center)}]{
	 \draw +(0,-.75) node[below] {\small $i$} .. controls (1,0) .. +(0,.75);
	 \draw +(1,-.75) node[below] {\small $j$} .. controls (0,0) .. +(1,.75);
	} =
	\tikz[very thick,baseline={([yshift=1.0ex]current bounding box.center)}]{
\node at (0.5,0) {$\cQ_{i,j}(Y_1,Y_2)$};
  \draw (-.55,-.4) -- (-.55,.4); \draw (1.55,-.4) -- (1.55,.4);
  \draw (-.55,-.4) -- (1.55,-.4); \draw (-.55,.4) -- (1.55,.4);
  \draw (0,-.75) node[below] {\small $i$} -- (0,-.4);\draw (0,.75) -- (0,.4);
  \draw (1,-.75) node[below] {\small $j$} -- (1,-.4);\draw (1,.75) -- (1,.4);
	}\qquad\text{if }i\neq j,
\\
	\tikz[very thick,baseline={([yshift=1.0ex]current bounding box.center)}]{
	  \draw (0,-.5) node[below] {\small $i$} -- (1,.5);
	  \draw (1,-.5) node[below] {\small $j$} -- (0,.5) node [near start,fill=black,circle,inner sep=2pt]{};
	}
	  =
	\tikz[very thick,baseline={([yshift=1.0ex]current bounding box.center)}]{
	  \draw (0,-.5) node[below] {\small $i$} -- (1,.5) ;
	  \draw (1,-.5) node[below] {\small $j$} -- (0,.5) node [near end,fill=black,circle,inner sep=2pt]{};
	}, \qquad
	\tikz[very thick,baseline={([yshift=1.0ex]current bounding box.center)}]{
	  \draw (0,-.5) node[below] {\small $i$} -- (1,.5) node [near start,fill=black,circle,inner sep=2pt]{};
	  \draw (1,-.5) node[below] {\small $j$} -- (0,.5);
	}
 =
	\tikz[very thick,baseline={([yshift=1.0ex]current bounding box.center)}]{
	  \draw (0,-.5) node[below] {\small $i$} -- (1,.5)node [near end,fill=black,circle,inner sep=2pt]{};
	  \draw (1,-.5) node[below] {\small $j$} -- (0,.5);
	} \qquad
	  \text{ if }i \ne j,
\\
\label{eq:nh1}
	\tikz[very thick,baseline={([yshift=1.0ex]current bounding box.center)}]{
	  \draw (0,-.5) node[below] {\small $i$} -- (1,.5);
	  \draw (1,-.5) node[below] {\small $i$} -- (0,.5) node [near end,fill=black,circle,inner sep=2pt]{};
	}
	 -
	\tikz[very thick,baseline={([yshift=1.0ex]current bounding box.center)}]{
	  \draw (0,-.5) node[below] {\small $i$} -- (1,.5);
	  \draw (1,-.5) node[below] {\small $i$} -- (0,.5) node [near start,fill=black,circle,inner sep=2pt]{};
	} =
	\tikz[very thick,baseline={([yshift=1.0ex]current bounding box.center)}]{
	  \draw (0,-.5) node[below] {\small $i$} -- (0,.5);
	  \draw (1,-.5) node[below] {\small $i$} -- (1,.5);
	}  =
 	\tikz[very thick,baseline={([yshift=1.0ex]current bounding box.center)}]{
	  \draw (0,-.5) node[below] {\small $i$} -- (1,.5) node [near start,fill=black,circle,inner sep=2pt]{};
	  \draw (1,-.5) node[below] {\small $i$} -- (0,.5);
	} 	 -
	\tikz[very thick,baseline={([yshift=1.0ex]current bounding box.center)}]{
	  \draw (0,-.5) node[below] {\small $i$} -- (1,.5)node [near end,fill=black,circle,inner sep=2pt]{};
	  \draw (1,-.5) node[below] {\small $i$} -- (0,.5);
		},
\\ \label{eq:R3_1}
	\tikz[very thick,scale=.75,baseline={([yshift=1.0ex]current bounding box.center)}]{
		 	 \draw +(1,-1) node[below] {\small $j$} .. controls (0,0) .. +(1,1);
		 \draw (0,-1) node[below] {\small $i$} -- (2,1);
	  \draw (0,1)-- (2,-1) node[below] {\small $k$} ;
	 }
=
	\tikz[very thick,scale=.75,baseline={([yshift=1.0ex]current bounding box.center)}]{
		 	 \draw +(1,-1) node[below] {\small $j$} .. controls (2,0) .. +(1,1);
		 \draw (0,-1) node[below] {\small $i$} -- (2,1);
	  \draw (0,1)-- (2,-1) node[below] {\small $k$};
	 }
\qquad \text{unless $i=k\neq j$,}
\\
 \label{eq:R3_2}
	\tikz[very thick,scale=.75,baseline={([yshift=1.0ex]current bounding box.center)}]{
		 	 \draw +(1,-1) node[below] {\small $j$} .. controls (0,0) .. +(1,1);
		 \draw (0,-1) node[below] {\small $i$} -- (2,1);
	  \draw (0,1)-- (2,-1) node[below] {\small $i$} ;
	 }   -
	 \tikz[very thick,scale=.75,baseline={([yshift=1.0ex]current bounding box.center)}]{
		 	 \draw +(1,-1) node[below] {\small $j$} .. controls (2,0) .. +(1,1);
		 \draw (0,-1) node[below] {\small $i$} -- (2,1);
	  \draw (0,1)-- (2,-1) node[below] {\small $i$};
	 }
 = \mspace{15mu}
	\tikz[very thick,scale=.75,baseline={([yshift=1.0ex]current bounding box.center)}]{
\node at (1,0) {\small $\dfrac{\cQ_{i,j}(Y_3,Y_2)-\cQ_{i,j}(Y_1,Y_2)}{Y_3-Y_1}$};
  \draw (-2,-.75) -- (-2,.75); \draw (3.9,-.75) -- (3.9,.75);
  \draw (-2,-.75) -- (3.9,-.75); \draw (-2,.75) -- (3.9,.75);
  \draw (0,-1) node[below] {\small $i$} -- (0,-.75);\draw (0,.75) -- (0,1);
  \draw (1,-1) node[below] {\small $j$} -- (1,-.75);\draw (1,.75) -- (1,1);
  \draw (2,-1) node[below] {\small $i$} -- (2,-.75);\draw (2,.75) -- (2,1);
	}
\qquad \text{if $i \neq j $.}
\end{gather}

\item[$\diamond$] And the \emph{additional relations, for all $i$, $j\in I$}:
\begin{gather}\label{eq:ExtR2}
	\tikz[very thick,baseline={([yshift=1.0ex]current bounding box.center)}]{
	 \draw (1,-.75) node[below] {\small $i$} -- (1,.75);
		\fdot[]{}{1.5,-0.25} \fdot[]{}{1.8,0.25}
\node at (2.75,0) {$\dots$};
	}
   =   0 ,
	\\
	\tikz[very thick,baseline={([yshift=1.0ex]current bounding box.center)}]{
	 \draw +(0,-.75) node[below] {\small $i$} .. controls (1.25,0) .. +(0,.75);
	 \draw +(1,-.75) node[below] {\small $j$} .. controls (-.25,0) .. +(1,.75);
	 \fdot{}{.5,0};\fdot{}{.5,.7};
  }
= -
	\tikz[very thick,baseline={([yshift=1.0ex]current bounding box.center)}]{
	 \draw +(0,-.75) node[below] {\small $i$} .. controls (1.25,0) .. +(0,.75);
	 \draw +(1,-.75) node[below] {\small $j$} .. controls (-.25,0) .. +(1,.75);
	 \fdot{}{.5,0};\fdot{}{.5,-.7};
  }. \label{eq:extrarel}
 \end{gather}
\end{itemize}\end{itemize}
\end{defn}

\begin{rem}
 A diagram with a box containing a polynomial means a polynomial in dots.
The indices in the variables indicate the strands carrying the corresponding dots.
 For example, for~$p(Y_1,Y_2)=\sum_{r,s} c_{r,s}Y_1^rY_2^s$ with $c_{r,s}\in\Bbbk$, we have
 \[ 	\tikz[very thick,baseline={([yshift=1.0ex]current bounding box.center)}]{
\node at (0.5,0) {$p(Y_1,Y_2)$};
  \draw (-.55,-.4) -- (-.55,.4); \draw (1.55,-.4) -- (1.55,.4);
  \draw (-.55,-.4) -- (1.55,-.4); \draw (-.55,.4) -- (1.55,.4);
  \draw (0,-.75) node[below] {\small $i$} -- (0,-.4);\draw (0,.75) -- (0,.4);
  \draw (1,-.75) node[below] {\small $j$} -- (1,-.4);\draw (1,.75) -- (1,.4);
 }
 = \sum_{r,s}c_{r,s}\
\tikz[very thick,baseline={([yshift=1.0ex]current bounding box.center)}]{
 \draw[black] (0,-.75) node[below] {\small $i$} -- (0,.75) node [midway,fill=black,circle,inner sep=2pt]{}; \node at (.25,0){$r$};
 \draw (1,-.75) node[below] {\small $j$} -- (1,.75) node [midway,fill=black,circle,inner sep=2pt]{}; \node at (1.25,0){$s$};	
	} .
 \]
\end{rem}

We now define a $\bZ\times\bZ$-grading in $\cR(\nu)$. Contrary to~\cite{naissevaz3}, we work with a
single homological degree $\lambda$. The homological nature of this degree is justified by the DG-structure defined in~Section~\ref{sec:dgKLR-diff}.
We declare
\[
\deg\Bigg(
 \tikz[very thick,baseline={([yshift=+.6ex]current bounding box.center)}]{
	  \draw (0,-.5)-- (0,.5) node [midway, fill=black,circle,inner sep=2pt]{};
		 \node at (0,-.75) {$i$};
 	} \Biggr) = (2,\ 0) ,
\qquad
\deg\Biggl(
\tikz[very thick,baseline={([yshift=+.6ex]current bounding box.center)}]{
	  \draw (-.5,-.5)-- (.5,.5);
	  \draw (.5,-.5)-- (-.5,.5);
		 \node at (-.5,-.75){$i$};
		 \node at (.5,-.75){$j$};
 	}
\Biggr) =
\begin{cases}
 (-2,0) & \text{if }\ i=j, \\
 (-1,0) & \text{if }\ h_{i,j} = 1, \\
 (0,0) & \text{otherwise},
\end{cases}
\]
and
\[
\deg\Bigg(\
\tikz[very thick,baseline={([yshift=.5ex]current bounding box.center)},decoration={
 markings, mark=at position 0.5 with {\arrow{>}}}] {
 \draw (-1,-.5) node[below] {\small $i$} -- (-1,.5);
 \node at (.5,0){$\cdots$};
	 \fdot[]{}{-.5,0};
	 	}
\Biggr) = \bigl( -2 ,\ 1 \bigr) ,
\]
where the second grading is called $\lambda$-grading, which we write $\lambda(\bullet)$.
The defining relations of~$\cR(\nu)$ are homogeneous with respect to this bigrading.

\begin{rem}
The subalgebra of $\cR(\nu)$ in $\lambda$-degree zero coincides with the usual KLR algebra~$R(\nu)$ defined in~\cite{KL1} and~\cite{R1}. More precisely, the algebra $R(\nu)$ is defined by the first two types of generators in Definition~\ref{def:KLR-enh} and relations \eqref{eq:KLRR2}--\eqref{eq:R3_2}.
\end{rem}

For $\bi=i_1\dots i_d$, define the idempotent
\[
1_{\bi}\, =
 \tikz[very thick,xscale=1.5,baseline={([yshift=.8ex]current bounding box.center)}]{
  \draw (-.5,-.5) node[below] {\small $i_1$} -- (-.5,.5);
  \draw (0,-.5) node[below] {\small $i_2$} -- (0,.5);
  \node at (0.75,0) {$\dots$};
  \draw (1.5,-.5) node[below] {\small $i_d$} -- (1.5,.5);
 }
\]
and let $\seq(\nu)$ be the set of all ordered sequences
$\bi = i_1i_2 \dots i_d$ with each $i_k \in I$ and $i$ appearing~$\nu_i$ times in the sequence.
For $\bi,\bj\in \seq(\nu)$ the idempotents $1_{\bi}$ and $1_{\bj}$ are orthogonal iff $\bi\neq\bj$,
we have $1_{\cR(\nu)} = \sum_{\bi\in \seq(\nu)}1_{\bi}$, where $1_{\cR(\nu)}$ denotes the identity element
in $\cR(\nu)$, and
\[
\cR(\nu) = \bigoplus\limits_{\bj,\bi\in \seq(\nu)} 1_{\bj}\cR(\nu)1_{\bi} .
\]
Finally, the algebra $\cR$ is defined as
\[
\cR = \bigoplus\limits_{\nu\in\bN[I]}\cR(\nu) .
\]

\subsection[Polynomial action of R(nu)]{Polynomial action of $\boldsymbol{\cR(\nu)}$}\label{ssec:polyaction}

We now describe a faithful action of $\cR(\nu)$ on a supercommutative ring, which was defined in~\cite[Section~3.2]{naissevaz3} and extends the polynomial action of KLR algebras from~\cite[Section~2.3]{KL1}.

 We fix $\nu \in \bN[I]$ with $|\nu| = d$.
 Set $PR_d=\Bbbk[Y_1, \dots, Y_d] \otimes \bV^\bullet \brak{\Omega_1, \dots, \Omega_{d}}$. Now consider
\begin{equation}
\label{eq:PRnu}
PR_\nu=\bigoplus_{\bi\in \seq(\nu)}PR_d 1_\bi.
\end{equation}
 Here we mean that the algebra $PR_\nu$ is a direct sum of copies of the algebra $PR_d$, labelled by~$\seq(\nu)$. We denote by $1_\bi$ the idempotent projecting to the $\bi$th copy.

For each $i\in I$, $1\leq r \leq \nu_i$ and $\bi=(i_1,i_2,\ldots,i_d)\in \Seq(\nu)$, we denote by $r'=r'(r,i,\bi)$ the $r$th index $r'\in\{1,2,\ldots,d\}$ (counting from the left) among the indices such that $i_{r'}=i$. Set~$\omega_{r,i}1_\bi=\Omega_{r'}1_\bi$.

The algebra $PR_\nu$ is bigraded supercommutative with gradings
$\deg(Y_t)=(2,0)$, $\deg(\omega_{r,i})=(-2r,1)$ and $\deg(1_\bi)=(0,0)$, where the variables $\omega_{r,i}$ are odd while the polynomial variables and the idempotents are even.
Note that we consider a $\lambda$-grading that is one half the one considered in~\cite{naissevaz3}.
 This is to agree with the analogous degrees on Hecke algebras in~Section~\ref{ssec:polyrings}.

Now, similarly to \cite[Section~3.2.1]{naissevaz3}, we consider the action of $\Sy_{|\nu|}$ on $PR_\nu$ given by
\[
s_k\colon\ PR_d 1_\bi \rightarrow PR_d 1_{s_k\bi},
\]
sends $Y_{p} 1_\bi \mapsto
Y_{s_k(p)}1_{s_k\bi}$ and
 \begin{align*}
\Omega_{p} 1_\bi &\mapsto \begin{cases}
\left( \Omega_{k} + (Y_{k} - Y_{k+1}) \Omega_{k+1} \right) 1_{\bi} &\text{ if $p = k$ and $i_k = i_{k+1}$}, \\
\Omega_{p} 1_{\bi} &\text{ if $p = k+1$ and $i_k = i_{k+1}$}, \\
\Omega_{s_k(p)} 1_{s_k \bi} &\text{ otherwise.}
\end{cases}
\end{align*}

For each $i,j\in I$, $i\ne j$, we consider the polynomial $\cP_{ij}(u,v)=(u-v)^{h_{i,j}}$, where $h_{i,j}$ denotes as above the number of arrows from $i$ to $j$. Note that we have $\cQ_{i,j}(u,v)=\cP_{i,j}(u,v)\cP_{j,i}(v,u)$.

In the sequel, it is useful to have an algebraic presentation of $\cR(\nu)$ as in~\cite[equations~(1.7)--(1.15)]{BK}.
We set
\[
\tikz[very thick,xscale=1.5,baseline={([yshift=.8ex]current bounding box.center)}]{
  \draw (-.4,-.5) node[below] {\small $i_1$} -- (-.4,.5);
  \node at (0,0) {$\dots$};
  \draw (0.4,-.5) node[below] {\small $i_r$} -- (0.4,.5)node [midway,fill=black,circle,inner sep=2pt]{};
 % \node at (0.4,0){\textbullet};
  \node at (0.8,0) {$\dots$};
  \draw (1.2,-.5) node[below] {\small $i_d$} -- (1.2,.5);
 } =\ Y_{r}1_{\bi} ,
\qquad
\tikz[very thick,baseline={([yshift=+.6ex]current bounding box.center)}]{
 \draw (-1.4,-.5) node[below] {\small $i_1$} -- (-1.4,.5);
 \draw (-.5,-.5)-- (.5,.5);
 \draw (.5,-.5)-- (-.5,.5);
 \node at (-.5,-.75){$i_r$};
 \node at (.5,-.75){$i_{r+1}$};
 \draw (1.4,-.5) node[below] {\small $i_d$} -- (1.4,.5);
 \node at (-.9,0) {$\dots$};\node at (.9,0) {$\dots$};
} = \tau_{r}1_{\bi} ,
\qquad
\tikz[very thick,baseline={([yshift=.5ex]current bounding box.center)},
 decoration={markings, mark=at position 0.5 with {\arrow{>}}}] {
 \draw (-1,-.5) node[below] {\small $i_1$} -- (-1,.5);
 \fdot[]{}{-.5,0};
 \draw (0,-.5) node[below] {\small $i_2$} -- (0,.5);
 \node at (.5,0) {$\dots$};
 \draw (1,-.5) node[below] {\small $i_d$} -- (1,.5);
} = \Omega 1_{\bi} .
 \]

We declare that $a \in e_{\bk}\cR(\nu) e_\bj$ acts as zero on $PR_I 1_\bi$ whenever $\bj \neq \bi$. Otherwise
\begin{align*}
&Y_r1_\bi
\longmapsto f1_\bi\mapsto Y_{r}f1_\bi ,
\qquad
\Omega 1_\bi \longmapsto
f1_\bi \mapsto \Omega_1f1_\bi ,
\end{align*}
and
\begin{align*}
&	\tau_r 1_\bi \longmapsto
f 1_\bi \mapsto
\begin{cases}
 \dfrac{f1_\bi - s_{r}(f1_\bi)}{Y_{r} - Y_{r+1}} & \text{if $i_r = i_{r+1}$},
 \\
 \cP_{i_r,i_{r+1}}(Y_{r},Y_{r+1})
 s_{r} (f 1_{\bi}) & \text{if $i_r \ne i_{r+1}$}.
 \\
\end{cases}
\end{align*}

The following is Proposition~3.8 and Theorem 3.15 in~\cite{naissevaz3}.
\begin{prop}
\label{prop:faithrep-KLR-Gr}
The rules above define a faithful action of $\cR(\nu)$ on $PR_\nu$.
\end{prop}

\subsection[Completion of R(nu)]{Completion of $\boldsymbol{\cR(\nu)}$}\label{ssec:complKLR}

We will consider $\PolR_d=\Bbbk[Y_1,Y_2,\ldots,Y_d]$ as a subalgebra of $\cR(\nu)$.
Let $\mathfrak{m}$ be the ideal of $\PolR_d$ generated by all $Y_p$, $1 \leqslant p\leqslant d$.

\begin{defn}
Denote by $\widehat \cR(\nu)$ the completion of the algebra $\cR(\nu)$ at the sequence of ideals~$\cR(\nu)\mathfrak{m}^j\cR(\nu)$. Let \smash{$\widehat{PR}_d=\Bbbk[[Y_1, \dots, Y_d]] \otimes \bV^\bullet \brak{\Omega_1, \dots, \Omega_{d}}$} be the similar completion of $PR_d$ and let \smash{$\widehat{PR}_\nu=\bigoplus_{\bi\in\Seq(\nu)}\widehat{PR}_d1_\bi$} be the similar completion of $PR_\nu$.
\end{defn}

We would like to construct a representation structure of $\widehat \cR(\nu)$ in the vector space $\widehat{PR}_\nu$.
The $\Sy_{|\nu|}$-action on $PR_\nu$ extends obviously to an $\Sy_{|\nu|}$-action on \smash{$\widehat{PR}_\nu$}. Moreover, the action of $\cR(\nu)$ on $PR_\nu$ yields an action of \smash{$\widehat \cR(\nu)$} on \smash{$\widehat{PR}_\nu$}.

\begin{lem}
 The representation $\widehat{PR}_\nu$ of $\widehat \cR(\nu)$ is faithful.
\end{lem}
\begin{proof}
 An explicit $\PolR_d$-basis of $\cR_\nu$ is constructed in \cite[Section~3.2]{naissevaz3}. We would like to check that the same set forms a \smash{$\widehat{\PolR}_d$}-basis of \smash{$\widehat \cR_\nu$}. The fact that this is a spanning set can be proved by the same argument. The linear independence follows from the fact that the elements act on \smash{$\widehat{PR}_\nu$} by linearly independent operators. Then, this proves automatically the faithfulness of the representation.
\end{proof}

%%%%%%%%%%%%%%%%%%%%%%%%%%%%%%%%%%%%%%%%%
\subsection{Cyclotomic KLR algebras}

Let $\Lambda$ be a dominant integral weight of type $\Gamma$ (i.e., for each vertex $i$ of $\Gamma$ we fix a nonnegative integer $\Lambda_i$). Let \smash{$I^\Lambda$} be the 2-sided ideal
of $R(\nu)$ generated by \smash{$Y_1^{\Lambda_{i_1}}1_\bi$} with $\bi\in\Seq(\nu)$. In terms of diagrams, this is the 2-sided ideal
generated by all diagrams of the form
\[
 \tikz[very thick,xscale=1.5,baseline={([yshift=.8ex]current bounding box.center)}]{
  \draw (-.5,-.5) node[below] {\small $i_1$} -- (-.5,.5) node [midway,fill=black,circle,inner sep=2pt]{}; \node at (-.75,0) {$\Lambda_{i_1}$};
  \draw (0,-.5) node[below] {\small $i_2$} -- (0,.5);
  \node at (.75,0){$\cdots$};
  \draw (1.5,-.5) node[below] {\small $i_{|\nu|}$} -- (1.5,.5);
 },
\]
with $\bi \in \Seq(\nu)$.

\begin{defn}
 The \emph{cyclotomic KLR algebra} is the quotient $R^\Lambda(\nu) = R(\nu)/ I^\Lambda$.
 \end{defn}

\subsection[DG-enhancements of R(nu)]{DG-enhancements of $\boldsymbol{\cR(\nu)}$}\label{sec:dgKLR-diff}

We turn $\cR(\nu)$ into a DG-algebra by introducing a differential $d_{\Lambda}$ given by
\begin{gather*}
d_\Lambda(1_\bi)=d_\Lambda(Y_r)=d_\Lambda(\tau_k)=0,
\qquad
d_\Lambda(\Omega 1_\bi)=(-Y_1)^{\Lambda_{i_1}}1_\bi,
\end{gather*}
together with the Leibniz rule
\[
d_\Lambda(ab)=d_\Lambda(a)b + (-1)^{\lambda(a)}d_\Lambda(b).
\]
This algebra is differential graded with respect to the homological degree given by counting the number of floating dots. Since $\frakm$ is in the kernel of $d_\Lambda$, we can extend $d_\Lambda$ to $\widehat \cR(\nu)$.

The following is~\cite[Proposition 4.14]{naissevaz3}.
\begin{prop}\label{prop:formalKLR}
The homology of the DG-algebra $(\cR(\nu),d_\Lambda)$ is concentrated in degree $0$ and is isomorphic to the cyclotomic KLR algebra $R^\Lambda(\nu)$.
\end{prop}

%%%%%%%%%%%%%%%%%%%%%%%%%
\section{The isomorphism theorems}\label{sec:isos}
%%%%%%%%%%%%%%%%%%%%%%%%%

\subsection{A generalization of the Brundan--Kleshchev--Rouquier isomorphisms}
Choose $I$, $\Gamma$ and $\nu$ as in~Section~\ref{sec:KLRG}. Assume additionally that for $i,j\in I$, $i\ne j$, there is at most one arrow from $i$ to $j$.

Let $\PolR_d$ be as in~Section~\ref{ssec:complKLR}. Set $\PolR_\nu=\bigoplus_{\bi\in\seq(\nu)}\PolR_d 1_\bi$. Here, similarly to \eqref{eq:PRnu}, the element $1_\bi$ is the idempotent projecting to the $\bi$th component of the direct sum.
Let $PA_\nu$ be a $\PolR_\nu$-algebra free over $\PolR_\nu$ (the most interesting examples for us are $PA_\nu=PR_\nu$ and~${PA_\nu=\PolR_\nu}$).
Set also \smash{$\widehat {PA}_\nu=\widehat\PolR_\nu\otimes_{\PolR_\nu}PA_\nu$}.

Fix an action of $\mathfrak{S}_{|\nu|}$ on $\widehat {PA}_\nu$ (by ring automorphisms) that extends the obvious $\mathfrak{S}_{|\nu|}$-action on $\widehat \PolR_\nu$. We assume that such an extension exists.
We make additionally the following assumption.

\begin{Assumption}\label{as1}
For each simple generator $s_r$ of $\mathfrak{S}_{|\nu|}$, each $\bi\in\seq(\nu)$ such that $i_r=i_{r+1}$ and each $f\in \widehat {PA}_\nu$, we have $(f-s_r(f))1_\bi\in (Y_r-Y_{r+1})\widehat {PA}_\nu$.
\end{Assumption}

This assumption implies that the Demazure operator $\frac{1-s_r}{Y_{r}-Y_{r+1}}$ is well defined on $\widehat {PA}_\nu 1_\bi$.
Fix a subalgebra \smash{$\widehat {PA}'_\nu$} of \smash{$\widehat {PA}_\nu$}. Assume now that we have an algebra \smash{$\widehat \cA(\nu)$} that has a faithful representation on $\widehat {PA}_\nu$.
We make the following assumption.

\begin{Assumption}\label{as2}
The action of $\widehat \cA(\nu)$ on $\widehat {PA}_\nu$ is generated by multiplication by elements of~$\widehat {PA}'_\nu$ and by the operators $\tau_r$, $r\in\{1,2,\ldots,|\nu|-1\}$ given by
\begin{itemize}\itemsep=0pt
\item if $i_r=i_{r+1}$, then $\tau_r$ acts on $f1_\bi$ by a (nonzero scalar) multiple of the Demazure operator, i.e., $\tau_r$ sends $f1_\bi$ to a multiple of $\frac{(f - s_{r}(f))1_\bi}{Y_r - Y_{r+1}}$,
\item if $i_r\ne i_{r+1}$, then $\tau_r$ sends $f1_\bi$ to $\cP_{i_r,i_{r+1}}(Y_r,Y_{r+1})s_{r} (f 1_{\bi})$.
\end{itemize}
\end{Assumption}

The goal for this section is to give non-trivial sufficient conditions for an algebra to be isomorphic to $\widehat \cA(\nu)$, generalizing the BKR isomorphism.

The table below summarizes the various rings appearing on the KLR side and on the Hecke side of the picture.
\begin{center}\renewcommand{\arraystretch}{1.3}
\begin{tabular}{|l|l|}
 \hline
 The KLR side & The Hecke side (degenerate version)
 \\ \hline
 $\PolR_\nu =\bigoplus\limits_{\bi\in\Seq(\nu)}\Bbbk[Y_1,\dots ,Y_d]1_{\bi}$
 &
 $\Pol_d =\Bbbk[X_1,\dots ,X_d]$
 \\ \hline
 $PA_\nu$: a $\PolR_\nu$-algebra & $PB_d$: a $\Pol_d$-algebra
 \\ \hline
 $\widehat{\PolR}_\nu =\bigoplus\limits_{\bi\in\Seq(\nu)}\Bbbk[[ Y_1,\dots ,Y_d]]1_{\bi}$
 &
 $\widehat{\Pol}_{\bfa} =\bigoplus\limits_{\bfb\in\Sy_d \bfa}\Bbbk[[X_1-b_1,\dots ,X_d-b_d]] 1_\bfb$
 \\ \hline
 $\widehat{PA}_\nu = \widehat{\PolR}_\nu\otimes_{\PolR_\nu}PA_\nu$
 &
 $\widehat{PB}_{\bfa} =\bigoplus\limits_{\bfb\in\Sy_d \bfa}\bigl(\Bbbk[[X_1-b_1,\dots ,X_d-b_d]]\otimes_{\Pol_d}PB_d \bigr) 1_\bfb$
 \\ \hline
 $\widehat{PA}'_\nu\subseteq \widehat{PA}_\nu$
  &
$\widehat{PB}'_{\bfa}\subseteq \widehat{PB}_{\bfa}$
 \\ \hline
 $\widehat \cA(\nu)=\langle \tau_r,\widehat{PA}'_\nu\rangle \subseteq \End(\widehat{PA}_\nu)$
  &
$\widehat{\bar \cB}_{\bfa}=\langle T_r,\widehat{PB}'_\bfa\rangle \subseteq \End(\widehat{PB}_\bfa)$
 \\ \hline
\end{tabular}
\end{center}
We have only included the degenerate version of the Hecke algebra in the column on the right, the $q$-version being very similar.

\subsubsection{Degenerate version}
\label{sssec:isom-deg-general}

Fix $\bfQ=(Q_1,\ldots,Q_\ell)\in \Bbbk^\ell$, as in~Section~\ref{ssec:DAH}. Now we fix some special choice of $\Gamma$ and $\nu$. Let~$I$ be a subset of $\Bbbk$ that contains $Q_1,\ldots,Q_\ell$.
We construct the quiver $\Gamma$ with the vertex set~$I$ using the following rule: for $i,j\in I$ we have an edge $i\to j$ if and only if we have $j+1=i$. Note that this convention for $\Gamma$ is opposite to \cite{R1}. Let $d$ be a positive integer. Fix $\bfa\in I^d$ (see~Section~\ref{ssec:compl_deg-Hecke-Gr}). Finally, we consider $\nu$ such that $\nu_i$ is the multiplicity of $i$ in $\bfa$. In particular, we see that $|\nu|=d$ is the length of $\bfa$. Note that we have $\Seq(\nu)=\Sy_d \bfa$.

For each $i\in I$, denote by $\Lambda_i$ the multiplicity of $i$ in $(Q_1,\ldots, Q_\ell)$. In particular, this implies~$\prod_{r=1}^\ell(X_1-Q_r)=\prod_{i\in I}(X_1-i)^{\Lambda_i}$.

As above, we set $\Pol_d=\Bbbk[X_1,\cdots,X_d]$. Let $PB_d$ be a $\Pol_d$-algebra free over $\Pol_d$. The most interesting examples are $PB_d=P_d$ and $PB_d=\Pol_d$. Set
 \begin{gather*}
\widehat\Pol_\bfa=\bigoplus_{\bfb\in\Sy_d\bfa}\Bbbk[[X_1-b_1,\ldots,X_d-b_d]]1_\bfb, \\
\widehat {PB}_\bfa=\bigoplus_{\bfb\in\Sy_d\bfa}(\Bbbk[[X_1-b_1,\ldots,X_d-b_d]]\otimes_{\Pol_d} PB_d)1_\bfb.
\end{gather*}
 Then $\widehat {PB}_\bfa$ is a $\widehat\Pol_\bfa$-algebra.

Fix an action of $\mathfrak{S}_{d}$ on $PB_d$ (by ring automorphisms) that extends the obvious $\mathfrak{S}_{d}$-action on~$\Pol_d$. We assume that such an extension exists. We assume additionally the following.

\begin{Assumption}\label{as3}
For each simple generator $s_r$ of $\mathfrak{S}_{d}$ and each $f\in PB_d$, we have \[f-s_r(f)\subseteq (X_r-X_{r+1}) PB_d.\]
\end{Assumption}

In particular, this assumption implies that the Demazure operator $\partial_r=\frac{1-s_r}{X_{r}-X_{r+1}}$ is well defined on $PB_d$. The action of $\Sy_d$ on $\Pol_d$ and $PB_d$ can be obviously extended to an action on $\widehat\Pol_\bfa$ and $\widehat{PB}_\bfa$.
Fix a subalgebra $\widehat{PB}'_\bfa$ of $\widehat{PB}_\bfa$. We make the following assumption.

\begin{Assumption}\label{as4}
There is an algebra $\widehat{\bar \cB}_\bfa$ that has a faithful representation in $\widehat{PB}_\bfa$ that is generated by multiplication by elements of $\widehat{PB}'_\bfa$ and by the operators $T_r=s_r-\partial_r$.
\end{Assumption}

By construction, we have the isomorphism
\begin{equation}
\label{eq:isom-Pol-degen}
\widehat \PolR_\nu\simeq \widehat \Pol_\bfa, \quad Y_r 1_\bi\mapsto (X_r-i_r) 1_\bi.
\end{equation}
Moreover, this isomorphism commutes with the action of $\Sy_{d}$. We assume the following.

\begin{Assumption}\label{as5}
 We can extend the isomorphism $\widehat\PolR_\nu\simeq \widehat \Pol_\bfa$ in \eqref{eq:isom-Pol-degen} to an $\Sy_d$-invariant isomorphism $\widehat {PA}_\nu\simeq\widehat {PB}_\bfa$. This extension restricts to an isomorphism $\widehat {RA}'_\nu\simeq\widehat {PB}'_\bfa$.
\end{Assumption}

We get the following proposition (if the Assumptions \ref{as1}--\ref{as5} are satisfied).
\begin{prop}
\label{prop:isom-deg-general}
There is an algebra isomorphism \smash{$\widehat \cA(\nu)\simeq \widehat {\bar \cB}_\bfa$} that intertwines the representation in \smash{$\widehat {PA}_\nu\simeq \widehat {PB}_\bfa$}.
\end{prop}

\begin{proof}
We only have to show that we can write the operator $\tau_r$ in terms of $T_r$ (and multiplication by elements of \smash{$\widehat {PA}'_\nu\simeq \widehat {PB}'_\bfa$}) and vice versa.

First of all, note that the element $(Y_r-Y_{r+1}+c)\in\Bbbk[[Y_1,\ldots,Y_d]]$ is invertible for each nonzero $c\in \Bbbk$ and that its inverse is $c^{-1}\big(\sum_{n\geq 0}c^{-n}(Y_{r+1}-Y_r)\big)$. Now, since we have \[(X_r-X_{r+1})1_\bi=(Y_r-Y_{r+1}+i_r-i_{r+1})\] under the isomorphism \smash{$\widehat\PolR_\nu\simeq \widehat \Pol_\bfa$}, we see that the element $(X_r-X_{r+1})^{-1}1_\bi\in \widehat \Pol_\bfa$ is well defined if $i_r\ne i_{r+1}$ and the element $(X_r-X_{r+1}+1)^{-1}1_\bi\in \Pol_\bfa$ is well defined if $i_r+1\ne i_{r+1}$.

First, we express $\tau_r$ in terms of $T_r$. We can rewrite the operator $T_r$ in the following way:
\[
T_r=1+\frac{X_r-X_{r+1}+1}{X_r-X_{r+1}}(s_r-1).
\]
Fix $\bi\in \seq(\nu)=\Sy_d \bfa$.
Assume $i_r=i_{r+1}$.
Then the action of the operator $(X_r-X_{r+1}+1)^{-1} 1_\bi$ on \smash{$\widehat {PA}_\nu\simeq\widehat {PB}_\bfa$} is well defined. The element $-(X_r-X_{r+1}+1)^{-1}(T_r-1)1_\bi$ acts on \smash{$\widehat {PB}_\bfa$} by the same operator as~$\tau_r 1_\bi$.

Now, assume that we have $i_r\ne i_{r+1}$. If additionally we have no arrow $i_{r}\to i_{r+1}$,
we can write $s_r 1_\bi=\big(\frac{X_r-X_{r+1}}{X_r-X_{r+1}+1}(T_r-1)+1\big) 1_\bi$. We need the condition $i_{r+1}+1\ne i_{r}$ to be able to divide by $(X_r-X_{r+1}+1)$ here. The operator $s_r 1_\bi$ acts on \smash{$\widehat {PA}_\nu\simeq \widehat {PB}_\bfa$} in the same way as $\tau_r 1_\bi$.
Finally, if we have $i_{r}\to i_{r+1}$, then the operator \[(X_r-X_{r+1}+1)s_r 1_\bi=[(X_r-X_{r+1})(T_r-1)+(X_r-X_{r+1}+1)] 1_\bi\] acts on $\widehat {PB}_\bfa$ in the same way as $\tau_r 1_\bi$.

Now, we express $T_r$ in terms of $\tau_r$. The operator $T_r 1_\bi$ acts by $\big[1+\frac{(X_r-X_{r+1}+1)}{X_r-X_{r+1}}(s_r-1)\big] 1_\bi$. In the case $i_r\ne i_{r+1}$, we are allowed to divide by $X_r-X_{r+1}$ here. If we additionally have no arrow $i_r\to i_{r+1}$, then the element $s_r 1_\bi$ acts in the same way as $\tau_r 1_\bi$. If we have an arrow $i_r\to i_{r+1}$, then $(X_r-X_{r+1}+1)s_r 1_\bi$ acts in the same way as $\tau_r 1_\bi$. It remains to treat the case $i_r=i_{r+1}$. In this case, the element $\frac{s_r-1}{X_r-X_{r+1}}$ acts in the same way as $-\tau_r 1_\bi$.
\end{proof}

\subsubsection[q-version]{$\boldsymbol{q}$-version}
\label{sssec:isom-q-general}

Fix $q\in\Bbbk$, $q\ne 0,1$. Fix also $\bfQ=(Q_1,\ldots,Q_\ell)\in (\Bbbk^\times)^\ell$, as in~Section~\ref{ssec:AH}. Now we fix some special choice of $\Gamma$ and $\nu$. Let $I$ be a subset of $\Bbbk^\times$ that contains $Q_1,\ldots,Q_\ell$. We construct the quiver $\Gamma$ with the vertex set $I$ using the following rule: for $i,j\in I$ we have an edge $i\to j$ if and only if we have $qj=i$. Note that this convention for $\Gamma$ is opposite to \cite{MaksimauStroppel} and \cite{R1}. Fix $\bfa\in I^d$ (see~Section~\ref{ssec:compl-affHeckGr}). Finally, we consider $\nu$ such that $\nu_i$ is the multiplicity of $i$ in $\bfa$. In particular, we see that $|\nu|=d$ is the length of $\bfa$. Note that we have $\Seq(\nu)=\Sy_d \bfa$.
As in the degenerate case, for each $i\in I$ we denote by $\Lambda_i$ the multiplicity of $i$ in $(Q_1,\ldots, Q_\ell)$.

Set \smash{$\Poll_d=\Bbbk\big[X^{\pm 1}_1,\cdots,X^{\pm 1}_d\big]$}. Let $PB_d$ be a $\Poll_d$-algebra, free over $\Poll_d$. The most interesting examples are $PB_d=P_d$ and $PB_d=\Poll_d$. Set $\widehat\Pol_\bfa=\bigoplus_{\bfb\in\Sy_d\bfa}\Bbbk[[X_1-b_1,\ldots,X_d-b_d]]1_\bfb$ and \smash{$\widehat {PB}_\bfa=\bigoplus_{\bfb\in\Sy_d\bfa}(\Bbbk[[X_1-b_1,\ldots,X_d-b_d]]\otimes_{\Poll_d} PB_d)1_\bfb$}. Then \smash{$\widehat{PB}_\bfa$} is a \smash{$\widehat\Pol_\bfa$}-algebra.

Fix an action of $\mathfrak{S}_{d}$ on $PB_d$ (by ring automorphisms) that extends the obvious $\mathfrak{S}_{d}$-action on $\Poll$. We assume additionally the following.

\begin{Assumption}\label{as6}
For each simple generator $s_r$ of $\mathfrak{S}_{d}$ and each $f\in PB_d$, we have \[f-s_r(f)\subseteq (X_r-X_{r+1}) Pl_d.\]
\end{Assumption}

In particular, this assumption implies that the Demazure operator $\frac{1-s_r}{X_{r}-X_{r+1}}$ is well defined on $Pl_d$. The action of $\Sy_d$ on $\Poll_d$ and $Pl_d$ can be obviously extended to an action on \smash{$\widehat\Poll_\bfa$} and~\smash{$\widehat {PB}_\bfa$}.

Fix a subalgebra $\widehat {PB}'_\bfa$ of $\widehat {PB}_\bfa$. We make the following assumption.

\begin{Assumption}\label{as7}
There is an algebra $\widehat \cB_\bfa$ that has a faithful representation in $\widehat {PB}_\bfa$ that is generated by multiplication by elements of \smash{$\widehat {PB}'_\bfa$} and by the operators
\[
T_r=q+\frac{(qX_r-X_{r+1})}{X_r-X_{r+1}}(s_r-1).
\]
\end{Assumption}

By construction, we have the isomorphism
\begin{equation}
\label{eq:isom-Pol-q}
\widehat \PolR_\nu\simeq \widehat \Pol_\bfa, \qquad Y_r 1_\bi\mapsto i_r^{-1}(X_r-i_r) 1_\bi.
\end{equation}
Moreover, this isomorphism commutes with the action of $\Sy_{d}$. We assume the following.

\begin{Assumption}\label{as8}
We can extend the isomorphism \smash{$\widehat\PolR_\nu\simeq \widehat \Pol_\bfa$} in \eqref{eq:isom-Pol-q} to an $\Sy_d$-invariant isomorphism \smash{$\widehat {PA}_\nu\simeq\widehat {PB}_\bfa$}. This extension restricts to an isomorphism \smash{$\widehat {PA}'_\nu\simeq\widehat {PB}'_\bfa$}.
\end{Assumption}

Then we have the following (if Assumptions~\ref{as1}, \ref{as2}, \ref{as6}, \ref{as7}, \ref{as8} are satisfied).

\begin{prop}
\label{prop:isom-q-general}
There is an algebra isomorphism $\widehat \cA(\nu)\simeq \widehat \cB_\bfa$ that intertwines the representation in \smash{$\widehat {PA}_\nu\simeq \widehat {PB}_\bfa$}.
\end{prop}
\begin{proof}
We only have to show that we can write the operator $\tau_r$ in terms of $T_r$ (and multiplication by elements of \smash{$\widehat {PA}'_\nu\simeq \widehat {PB}'_\bfa$}) and vice versa.
First, we express $\tau_r$ in terms of $T_r$.
Fix $\bi\in \seq(\nu)=\Sy_d \bfa$.

Assume $i_r=i_{r+1}$.
Then the action of the operator $(qX_r-X_{r+1})^{-1} 1_\bi$ on $\widehat {PA}_\nu\simeq\widehat {PB}_\bfa$ is well defined. The element $-(qX_r-X_{r+1})^{-1}(T_r-q)1_\bi$ acts on \smash{$\widehat {PB}_\bfa$} by the same operator as~$\tau_r 1_\bi$.

Now, assume that we have $i_r\ne i_{r+1}$. If moreover we have no arrow $i_{r}\to i_{r+1}$, we can write $s_r 1_\bi=\big(\frac{X_r-X_{r+1}}{qX_r-X_{r+1}}(T_r-q)+1\big) 1_\bi$ (we need the condition $qi_{r+1}\ne i_{r}$ to be able to divide by $(qX_r-X_{r+1})$ here). The operator $s_r 1_\bi$ acts on \smash{$\widehat {PA}_\nu\simeq \widehat {PB}_\bfa$} in the same way as $\tau_r 1_\bi$.
Finally, if we have $i_{r}\to i_{r+1}$, then the operator $(qX_r-X_{r+1})s_r 1_\bi=[(X_r-X_{r+1})(T_r-q)+(qX_r-X_{r+1})] 1_\bi$ acts on \smash{$\widehat {PB}_\bfa$} in the same way as $\tau_r 1_\bi$ up to scalar.

Now, we express $T_r$ in terms of $\tau_r$. The operator $T_r 1_\bi$ acts by $\big[q+\frac{(qX_r-X_{r+1})}{X_r-X_{r+1}}(s_r-1)\big] 1_\bi$. In the case $i_r\ne i_{r+1}$, we are allowed to divide by $X_r-X_{r+1}$ here. If we additionally have no arrow~${i_r\to i_{r+1}}$, then the element $s_r 1_\bi$ acts in the same way as $\tau_r 1_\bi$. If we have an arrow~${i_r\to i_{r+1}}$, then $(qX_r-X_{r+1})s_r 1_\bi$ acts up to scalar in the same way as $\tau_r 1_\bi$. It remains to treat the case~${i_r=i_{r+1}}$. In this case, the element $\frac{s_r-1}{X_r-X_{r+1}}$ acts in the same way as $-\tau_r 1_\bi$.
\end{proof}

\subsection{The DG-enhanced isomorphism theorem: the degenerate version}
\label{subs:DG-isom-degen}

In~Proposition~\ref{prop:isom-deg-general}, we proved that we have an isomorphism of algebras \smash{$\widehat\cA(\nu)\simeq \widehat {\bar \cB}_\bfa$} for some algebras \smash{$\widehat \cA(\nu)$} and \smash{$\widehat {\bar \cB}_\bfa$} that satisfy some list of properties. Let us show that we can apply~Proposition~\ref{prop:isom-q-general} to the special situation \smash{$\widehat \cA(\nu)=\widehat \cR(\nu)$} and \smash{$\widehat {\bar \cB}_\bfa=\widehat {\bar \cH}_\bfa$}. We assume that $\nu$ and $\bfa$ are related as in~Section~\ref{sssec:isom-deg-general}. In this case we can take \smash{$\widehat {PA}_\nu=\widehat {PR}_\nu$} and \smash{$\widehat {PB}_\bfa=\widehat {P}_\bfa$}. We consider the subalgebra \smash{$\widehat {PA}'_\nu$} of \smash{$\widehat {PA}_\nu$} generated by \smash{$\widehat\PolR_\nu$} and \smash{$\Omega_1$}, and the subalgebra \smash{$\widehat {PB}'_\bfa$} of \smash{$\widehat {PB}_\bfa$} generated by \smash{$\widehat\Pol_\bfa$} and $\theta_1$.

To be able to apply~Proposition~\ref{prop:isom-deg-general}, we only have to construct a $\Sy_d$-invariant isomorphism~\smash{${\alpha\colon \widehat {P}_\bfa\simeq\widehat {PR}_\nu}$} extending the isomorphism \eqref{eq:isom-Pol-degen} such that $\alpha$ restricts to an isomorphism \smash{$\widehat {PB}'_\bfa\simeq\widehat {PA}'_\nu$}. First, we consider the following homomorphism $\alpha'\colon\widehat \Pol_\bfa\to\widehat {PR}_\nu$.
\begin{gather*}
1_\bi\mapsto 1_\bi,\qquad
X_r1_\bi\mapsto (Y_r+i_r)1_\bi.
\end{gather*}
This homomorphism is obviously $\Sy_d$-invariant.

\begin{rem}
For each $1\leqslant r<d$, the Demazure operator $\partial_r=\frac{1-s_r}{X_r-X_{r+1}}$ is well defined on~\smash{$\widehat P_\bfa$}. Now, using the isomorphism \smash{$\widehat \Pol_\bfa\simeq\widehat {\PolR}_\nu$}, we can consider it as an operator on \smash{$\widehat {PR}_\nu$}. The action of $\partial_r$ on \smash{$\widehat {PR}_\nu$} can be given explicitly by
\[
\partial_r(f1_\bi)=\frac{f1_\bi-s_r(f)1_{s_r(\bi)}}{Y_r-Y_{r+1}+i_r-i_{r+1}}, \qquad f\in\Bbbk[[Y_1,\ldots,Y_d]].
\]
Attention, the operator $\partial_r$ on $\widehat {PR}_\nu$ should not be confused with $\frac{1-s_r}{Y_r-Y_{r+1}}$, which is not well defined.
The Demazure operators $\partial_r$ on \smash{$\widehat {PR}_\nu$} satisfy relations~\eqref{eq:DemazureR23}, \eqref{eq:DemazureComm}, \eqref{eq:DemazureX}.
\end{rem}

Now, we want to extend $\alpha'$ to a homomorphism $\alpha\colon \widehat P_\bfa\to\widehat {PR}_\nu$. To do this, we have to choose the images of $\theta_1,\theta_2,\ldots,\theta_d$ in $\widehat {PR}_\nu$ such that these images anticommute with each other and commute with the image of $\widehat\Pol_\bfa$ (i.e., with $\widehat \PolR_\nu$). Moreover, we want to make this choice in such a way that $\alpha$ is bijective and $\Sy_d$-invariant.

First, we set
\begin{equation}
\label{eq:base-theta1}
\alpha(\theta_11_\bi)= \bigg(\prod_{i\in I,i\ne i_1}(Y_1+i_1-i)^{\Lambda_{i}}\bigg)(-1)^{\Lambda_{i_1}}\Omega_11_\bi.
\end{equation}
This choice is motivated by the fact that we will want $\alpha$ to be compatible with the DG-structure. For $r>1$, we construct the images of other $\theta_r$ in the following way
\begin{equation}
\label{eq:induc-thetar}
\alpha(\theta_{r})=(-1)^{r-1}\partial_{r-1}\cdots\partial_{2}\partial_{1}(\alpha(\theta_{1})).
\end{equation}
This choice is motivated by the fact that we want $\alpha$ to be $\Sy_d$-invariant and we have that $\theta_{r}=-\partial_{r-1}(\theta_{r-1})$.
Since we have $s_r=1-(X_r-X_{r+1})\partial_r$, equation \eqref{eq:induc-thetar} implies immediately
\begin{equation}
\label{eq:comm-sr-thr}
\alpha(s_r(\theta_r))=s_r(\alpha(\theta_r)).
\end{equation}

\begin{lem}
\label{lem:alpha-isom}
The homomorphism $\alpha\colon \widehat P_\bfa\to\widehat {PR}_\nu$ given by \eqref{eq:base-theta1} and \eqref{eq:induc-thetar} is an isomorphism and it is $\Sy_d$-invariant.
\end{lem}
\begin{proof}
Since the homomorphism $\alpha'\colon\widehat \Pol_\bfa\to\widehat {PR}_\nu$ is obviously $\Sy_d$-invariant, to show the $\Sy_d$-invariance of $\alpha$, we have to show
\begin{equation}
\label{eq:Sd-invar}
s_k(\alpha(\theta_r1_\bi))=\alpha(s_k(\theta_r1_\bi))
\end{equation}
for each $\bi\in \Seq(\nu)$, each $r\in[1;d]$ and each $k\in[1;d-1]$.
We give a proof by induction on $r$. First, we prove \eqref{eq:Sd-invar} for $r=1$. If $k>1$ and $r=1$, then \eqref{eq:Sd-invar} is obvious because $\theta_1$ and $\alpha(\theta_1)$ are $s_k$-invariant. The case $k=r=1$ follows from \eqref{eq:comm-sr-thr}.

Now, assume that $r>1$ and that \eqref{eq:Sd-invar} is already proved for smaller values of $r$. The case~${k=r}$ follows from \eqref{eq:comm-sr-thr}.

For $k\ne r$, the element $\theta_r$ is $s_k$-invariant. So \eqref{eq:Sd-invar} is equivalent to the $s_k$-invariance of $\alpha(\theta_r)$.

Assume that $k>r$ or $k<r-2$. This assumption implies that $s_k$ commutes with $s_{r-1}$. Moreover, we already know by induction hypothesis that $\alpha(\theta_{r-1})$ is $s_{k}$-invariant. So, the $s_k$-invariance of $\alpha(\theta_{r-1})$ together with \eqref{eq:induc-thetar} implies the $s_k$-invariance of $\alpha(\theta_r)$.

Now, assume $k=r-1$. In this case the $s_{r-1}$-invariance of $\alpha(\theta_r)$ is obvious from \eqref{eq:induc-thetar}.

Finally, assume $k=r-2$. To prove the $s_{r-2}$-invariance of $\alpha(\theta_r)$, we have to show that $\partial_{r-2}(\alpha(\theta_r))=0$.
We have
\[
\partial_{r-2}(\alpha(\theta_r))=\partial_{r-2}\partial_{r-1}\partial_{r-2}(\alpha(\theta_{r-2}))=\partial_{r-1}\partial_{r-2}\partial_{r-1}(\alpha(\theta_{r-2})).
\]
This is equal to zero because $\partial_{r-1}(\alpha(\theta_{r-2}))=0$ by the $s_{r-1}$-invariance of $\alpha(\theta_{r-2})$.

This completes the proof of the $\Sy_d$-invariance of $\alpha$.

Now, let us prove that $\alpha$ is an isomorphism. It is easy to see from \eqref{eq:base-theta1} and \eqref{eq:induc-thetar} that $\alpha(\theta_r1_\bi)$ is of the form
\begin{equation}
\label{eq:triang-theta-Omega}
\alpha(\theta_r1_\bi)=\sum_{t=1}^rP_t\Omega_t1_\bi,
\end{equation}
where $P_t\in \widehat {PR}_\nu1_\bi$ for $r\in\{1,2,\ldots,r\}$ and $P_r$ is invertible in $\widehat {PR}_\nu1_\bi$. Then the bijectivity is clear from \eqref{eq:triang-theta-Omega} and from the fact that $\alpha$ restricts to a bijection
$\widehat \Pol_\bfa\simeq \widehat\PolR_\nu$.
\end{proof}

We get the following theorem.
\begin{thm}\label{thm:isoEDHecke}
There is an isomorphism of DG-algebras $(\widehat \cR(\nu),d_\Lambda)\simeq (\widehat {\bar \cH}_\bfa,\partial_\bfQ)$.
\end{thm}

\begin{proof}
Note that \eqref{eq:base-theta1} implies that the isomorphism $\alpha$ (see Lemma~\ref{lem:alpha-isom}) identifies the subalgebra~\smash{$\widehat {PA}'_\nu$} of \smash{$\widehat {PA}_\nu$} with the subalgebra \smash{$\widehat {PB}'_\bfa$} of \smash{$\widehat {PB}_\bfa$}. Then the isomorphism of algebras follows immediately from~Proposition~\ref{prop:isom-deg-general}. We only have to check the DG-invariance.

Denote by $\gamma$ the isomorphism of algebras \smash{$\gamma\colon \widehat {\bar \cH}_\bfa\to \widehat \cR(\nu)$}. It is obvious that $\gamma$ preserves the $\lambda$-grading.
We claim that for each \smash{$h\in \widehat {\bar \cH}_\bfa$}, we have
\begin{equation}
\label{eq:dg-invar}
\gamma(\partial_\bfQ(h))=d_\Lambda(\gamma(h)).
\end{equation}
Indeed, it is enough to check \eqref{eq:dg-invar} for $h=\theta$. This follows directly from \eqref{eq:base-theta1}. In fact, this is exactly the reason why we define \eqref{eq:base-theta1} in such a way.
\end{proof}

\begin{rem}
We could also take \smash{$\widehat {PA}_\nu=\widehat {PA}'_\nu=\widehat {\PolR}_\nu$} and \smash{$\widehat {PB}_\bfa=\widehat {PB}'_\bfa=\widehat \Pol_\bfa$}. Then we get (the completion version of) the usual Brundan--Kleshchev--Rouquier isomorphism.
\end{rem}

\subsection[The DG-enhanced isomorphism theorem: the q-version]{The DG-enhanced isomorphism theorem: the $\boldsymbol{q}$-version}
\label{sec:isoqversion}

In~Proposition~\ref{prop:isom-deg-general}, we proved that we have an isomorphism of algebras \smash{$\widehat \cA(\nu)\simeq \widehat {\cB}_\bfa$} for some algebras \smash{$\widehat \cA(\nu)$} and \smash{$\widehat {\cB}_\bfa$} that satisfy some list of properties. Let us show that we can apply~Proposition~\ref{prop:isom-q-general} to the special situation \smash{$\widehat \cA(\nu)=\widehat \cR(\nu)$} and \smash{$\widehat {\cB}_\bfa=\widehat {\cH}_\bfa$}. We assume that $\nu$ and $\bfa$ are related as in~Section~\ref{sssec:isom-q-general}. In this case, we can take \smash{$\widehat {PA}_\nu=\widehat {PR}_\nu$} and \smash{$\widehat {PB}_\bfa=\widehat P_\bfa$}.

To be able to apply~Proposition~\ref{prop:isom-q-general}, we only have to construct a $\Sy_d$-invariant
isomorphism \smash{$\alpha\colon\widehat {PR}_\nu\simeq \widehat P_\bfa$} extending the isomorphism \eqref{eq:isom-Pol-q} such that $\alpha$ restricts to an isomorphism \smash{$\widehat {PA}'_\nu\simeq \widehat {PB}'_\bfa$} (we choose the subalgebras \smash{$\widehat {PA}'_\nu\subseteq \widehat {PA}_\nu$} and \smash{$\widehat {PB}'_\bfa\subseteq\widehat {PB}_\bfa$} in the same way as in~Section~\ref{subs:DG-isom-degen}). This can be done in the same way as in the degenerate case. However, some formulas in this case are different from the previous section because of the difference between~\eqref{eq:isom-Pol-degen} and~\eqref{eq:isom-Pol-q}. Here, we only give the modified formulas. The proofs are the same as in the previous section.

We consider the $\Sy_d$-invariant homomorphism $\alpha'\colon\widehat \Pol_\bfa\to\widehat {PR}_\nu$
\begin{gather*}
1_\bi\mapsto 1_\bi,\qquad
X_r1_\bi\mapsto i_r(Y_r+1)1_\bi.
\end{gather*}
Now, we extend $\alpha'$ to a homomorphism $\alpha\colon \widehat P_\bfa\to\widehat {PR}_\nu$ in the following way:
\begin{gather*}
\alpha(\theta_11_\bi)= \bigg(\prod_{i\in I,\,i\ne i_1}(i_1(Y_1+1)-i)^{\Lambda_{i}}\bigg)(-i_1)^{\Lambda_{i_1}}\Omega_11_\bi,\\ %\label{eq:base-theta1-q}\\
\alpha(\theta_{r})=(-1)^{r-1}\partial_{r-1}\cdots\partial_{2}\partial_{1}(\alpha(\theta_{1})).%\label{eq:induc-thetar-q}
\end{gather*}
As in the previous section, we can show that $\alpha$ is a $\Sy_d$-invariant isomorphism.

We get the following theorem.
\begin{thm}\label{thm:isoEAHecke}
There is an isomorphism of DG-algebras $\big(\widehat \cR(\nu),d_\Lambda\big)\simeq \big(\widehat {\cH}_\bfa,\partial_\bfQ\big)$.
\end{thm}

\begin{rem}
We could also take \smash{$\widehat {PA}_\nu=\widehat {PA}'_\nu=\widehat {\PolR}_\nu$} and \smash{$\widehat {PB}_\bfa=\widehat {PB}'_\bfa=\widehat \Pol_\bfa$}. Then we get (the completion version of) the usual Brundan--Kleshchev--Rouquier isomorphism.
\end{rem}

\subsection[The homology of bar H\_d and H\_d]{The homology of $\boldsymbol{\bar{\cH}_d}$ and $\boldsymbol{\cH_d}$}\label{sec:formalHecke}

We now have the tools to prove the following two propositions.
\begin{prop}\label{prop:formalDHecke}
 The homology of the DG-algebra $\big(\widebar{\cH}_d,\partial_\bfQ\big)$ is concentrated in degree $0$ and is isomorphic to
 \smash{$\widebar{H}_d^\bfQ$}.
\end{prop}

\begin{prop}\label{prop:formalAHecke}
The homology of the DG-algebra $(\cH_d,\partial_\bfQ)$ is concentrated in degree $0$ and is isomorphic to \smash{$H_d^\bfQ$}.
\end{prop}

First, we start from a similar statement for the KLR algebra.

\begin{prop}
\label{prop:formalKLRcomp}
The homology of the DG-algebra $\big(\widehat\cR(\nu),d_\Lambda\big)$ is concentrated in degree $0$ and is isomorphic to $R^\Lambda(\nu)$.
\end{prop}
\begin{proof}
It is proved in \cite[Proposition 4.14]{naissevaz3} that the homology of the DG-algebra $(\cR(\nu),d_\Lambda)$ is concentrated in degree $0$ and is isomorphic to $R^\Lambda(\nu)$. The same proof with minor modifications applies to our case. We just have to replace polynomials by power series.
\end{proof}

\begin{cor}\label{cor:formal-Hecke}
The homologies of the DG-algebras \smash{$\big(\widehat {\bar \cH}_\bfa,\partial_\bfQ\big)$} and $\big(\widehat {\cH}_\bfa,\partial_\bfQ\big)$ are concentrated in degree $0$ and are isomorphic to $\bar{H}_\bfa^\bfQ$ and ${H}_\bfa^\bfQ$, respectively.
\end{cor}

\begin{proof}
 The statement follows from Theorems~\ref{thm:isoEDHecke} and~\ref{thm:isoEAHecke}, Proposition~\ref{prop:formalKLRcomp} and from the usual Brundan--Kleshchev--Rouquier isomorphism.
\end{proof}

\begin{proof}[Proof of~Propositions~\ref{prop:formalDHecke} and~\ref{prop:formalAHecke}]
It is obvious that the homology group of\smash{ $\big(\bar{\cH}_d,\partial_\bfQ\big)$} in degree zero is $\bar{H}^\bfQ_d$. We only have to check that the homology groups in other degrees are zero.

Assume, that for some $i>0$, we have $H^i\big(\bar{\cH}_d,\partial_\bfQ\big)\ne 0$ and consider it as a $\Pol_d$-module. The annihilator of this $\Pol_d$-module is contained in some maximal ideal $\mathcal M\subseteq\Pol_d$. The ideal $\mathcal M$ is of the form $\mathcal M=(X_1-a_1,\ldots,X_d-a_d)$ for some $\bfa=(a_1,\ldots,a_d)\in\Bbbk^d$.

Then the completion of $H^i\big(\bar{\cH}_d,\partial_\bfQ\big)\ne 0$ with respect to the ideal $\mathcal M$ is nonzero.
This leads to a contradiction because \smash{$H^i\big(\widehat {\bar \cH}_\bfa,\partial_\bfQ\big)=0$} together with K\"unneth formula implies
\[
\Bbbk[[X_1-a_1,\ldots,X_d-a_d]]\otimes_{\Pol_d}H^i\big(\bar{\cH}_d,\partial_\bfQ\big)= 0 .
\]
Proposition~\ref{prop:formalAHecke} is proved in the same way.
\end{proof}

\subsection*{Acknowledgements}

We thank Jonathan Grant for useful discussions and the anonymous referees for the careful reading of our document.
PV was supported by the Fonds de la Recherche Scientifique -- FNRS under Grant no.~MIS-F.4536.19.

\pdfbookmark[1]{References}{ref}
\LastPageEnding

\end{document}

%% file: defs.tex
\definecolor{myblue}{rgb}{0,.5,1}
\definecolor{myred}{rgb}{0.9,0,0}
\definecolor{mygreen}{rgb}{0,0.7,0}

\newcommand{\Sy}{\mathfrak{S}}
\newcommand{\bi}{{\boldsymbol{i}}}
\newcommand{\bj}{{\boldsymbol{j}}}
\newcommand{\bk}{{\boldsymbol{k}}}

\newcommand{\bV}{\raisebox{0.03cm}{\mbox{\footnotesize$\textstyle{\bigwedge}$}}}

\newcommand{\und}[1]{{\underline{#1}}}

\newcommand{\cH}{\mathcal{H}}
\newcommand{\cP}{\mathcal{P}}
\newcommand{\brak}[1]{\langle #1\rangle}

\DeclareMathOperator{\Seq}{Seq}

\DeclareMathOperator{\End}{End}

\DeclareMathOperator{\seq}{Seq}

\newtheorem{thm}{Theorem}[section]
\newtheorem{lem}[thm]{Lemma}
\newtheorem{cor}[thm]{Corollary}
\newtheorem{prop}[thm]{Proposition}

\theoremstyle{definition}
\newtheorem{defn}[thm]{Definition}

\newtheorem{rem}[thm]{Remark}
\newtheorem{Assumption}[thm]{Assumption}
\newtheorem{Questions}[thm]{Questions}

\newcommand{\bN}{\mathbb{N}}
\newcommand{\bZ}{\mathbb{Z}}

\newcommand{\cA}{\mathcal{A}}
\newcommand{\cB}{\mathcal{B}}

\newcommand{\cQ}{\mathcal{Q}}
\newcommand{\cR}{\mathcal{R}}

\DeclareMathOperator{\Pol}{Pol}
\DeclareMathOperator{\Poll}{Poll}
\DeclareMathOperator{\PolR}{PolR}
\DeclareMathOperator{\Sym}{Sym}
\DeclareMathOperator{\Syml}{Syml}

\newcommand{\bfa}{\mathbf{a}}
\newcommand{\bfb}{\mathbf{b}}
\newcommand{\bfQ}{\mathbf{Q}}
\newcommand{\frakm}{\mathfrak{m}}

\catcode`\@=11
\long\def\@makecaption#1#2{%
    \vskip 10pt
    \setbox\@tempboxa\hbox{%
\small{#1: }\ignorespaces #2}%
    \ifdim \wd\@tempboxa >\captionwidth {%
        \rightskip=\@captionmargin\leftskip=\@captionmargin
        \unhbox\@tempboxa\par}%
      \else
        \hbox to\hsize{\hfil\box\@tempboxa\hfil}%
    \fi}
\newdimen\@captionmargin\@captionmargin=2\parindent
\newdimen\captionwidth\captionwidth=\hsize
\catcode`\@=12
\def\makeautorefname#1#2{\expandafter\def\csname#1autorefname\endcsname{#2}}
\makeautorefname{lem}{Lemma}%
\makeautorefname{prop}{Proposition}%
\makeautorefname{rem}{Remark}%
\makeautorefname{section}{Section}%
\makeautorefname{subsection}{Section}%
\makeautorefname{subsubsection}{Section}%